\documentclass[12pt]{iopart}
\usepackage{lipsum}
\usepackage{amsfonts}
\usepackage{amsthm}
\usepackage{graphicx}
\usepackage{epstopdf}
\usepackage{algorithmic}
\ifpdf
  \DeclareGraphicsExtensions{.eps,.pdf,.png,.jpg}
\else
  \DeclareGraphicsExtensions{.eps}
\fi

\usepackage{color}

\usepackage{enumitem}
\setlist[enumerate]{leftmargin=.5in}
\setlist[itemize]{leftmargin=.5in}

\newtheorem{theorem}{Theorem}[section]

\newtheorem{lemma}[theorem]{Lemma}

\newtheorem{corollary}[theorem]{Corollary}

\newtheorem{proposition}[theorem]{Proposition}

\newtheorem{remark}[theorem]{Remark}

\newcommand{\cA}{\mathcal{A}}

\newcommand{\cD}{\mathcal{D}}

\newcommand{\cF}{\mathcal{F}}
\newcommand{\cG}{\mathcal{G}}
\newcommand{\cH}{\mathcal{H}}

\newcommand{\cL}{\mathcal{L}}

\newcommand{\cN}{\mathcal{N}}
\newcommand{\cO}{\mathcal{O}}

\newcommand{\cR}{\mathcal{R}}
\newcommand{\cS}{\mathcal{S}}

\newcommand{\cX}{\mathcal{X}}

\newcommand{\E}{\mathbb{E}}

\newcommand{\N}{\mathbb{N}}

\newcommand{\R}{\mathbb{R}}

\newcommand{\PP}{\mathbb{P}}
\renewcommand{\d}{\mathrm d}
\usepackage{amsopn}
\DeclareMathOperator{\diag}{diag}

\newcommand{\trace}{\operatornamewithlimits{trace}}

\usepackage[normalem]{ulem}

\newcommand{\cls}[1]{{\color{black}{#1}}}
\newcommand{\sw}[1]{{\color{black}{#1}}}

\newcommand{\db}[1]{{\color{black}{#1}}}

\usepackage{subfigure}

\usepackage{iopams}  
\begin{document}

\title{Well Posedness and Convergence Analysis of the Ensemble Kalman Inversion}

\author{Dirk Bl\"omker$^1$, Claudia Schillings$^2$, Philipp Wacker$^3$ and Simon Weissmann$^2$}

\address{$^1$ Universit\"at Augsburg, $^2$ Universit\"at Mannheim, $^3$ Universit\"at Erlangen}
\eads{\mailto{dirk.bloemker@math.uni-augsburg.de}, \mailto{c.schillings@uni-mannheim.de}, \mailto{phkwacker@gmail.com}, \mailto{sweissma@mail.uni-mannheim.de}}

\vspace{10pt}
\begin{indented}
\item[]September 2018
\end{indented}

\begin{abstract}
The ensemble Kalman inversion is widely used in practice to estimate unknown parameters from noisy measurement data. Its low computational costs, straightforward implementation, and non-intrusive nature makes the method appealing in various areas of application. We present a complete analysis of the ensemble Kalman inversion with perturbed observations for a fixed ensemble size when applied to linear inverse problems. The well-posedness and convergence results are based on the continuous time scaling limits of the method. The resulting coupled system of stochastic differential equations allows to derive estimates on the long-time behaviour and provides insights into the convergence properties of the ensemble Kalman inversion. We view the method as a derivative free optimization method for the least-squares misfit functional, which opens up the perspective to use the method in various areas of applications such as imaging, groundwater flow problems, biological problems as well as in the context of the training of neural networks. 
\end{abstract}

%
\vspace{2pc}

\submitto{\IP}
%
%
%
\section{Introduction}

Inverse problems arise in various fields of sciences and engineering. Methods to efficiently incorporate data into models are needed 
to reduce the overall uncertainty and to ensure the reliability of the simulations under real world conditions. 
The Bayesian approach to inverse problems provides a rigorous framework for the incorporation and quantification of 
uncertainties in measurements, parameters and models. However, in computationally intense applications, 
the approximation of the solution of the Bayesian inverse problem, the posterior, might be prohibitively expensive. 
In such settings, the ensemble Kalman filter (EnKF), originally introduced by Evensen \cite{Evensen2003} for data assimilation, 
has been reported to produce reliable estimates of the unknown parameters with low computational cost, 
making the method very appealing for large scale problems. \cls{Areas of applications include, among others, 
groundwater flow problems \cite{oliver2008inverse}, climate models \cite{doi:10.1002/2017GL076101}, biological problems \cite{HU2012145}, image reconstruction \cite{4914783} and building \cite{DESIMON2018220} and material sciences \cite{Iglesias_2018}.}
Most recent directions involve the use of the ensemble Kalman inversion as a derivative free optimization method, 
in particular in the context of the training of neural networks \cite{Nikola}. Despite its documented success, 
the ensemble Kalman inversion is underpinned by limited theoretical understanding. 
The goal of our work is to give useful insights into properties of the method and provide tools for a systematic development and improvement.

For linear dynamical systems and Gaussian initial conditions, 
analysis of the large ensemble size limit has been done e.g.\ in \cite{L2009LargeSA,doi:10.1137/140965363}. 
Convergence to the mean-field Kalman filter for nonlinear systems can be found in \cite{doi:10.1137/140984415}. 
Multilevel extensions are proposed e.g.\ in \cite{doi:10.1137/15M100955X, 2016arXiv160808558C}. 
In \cite{0951-7715-27-10-2579,47c5c78ebb8c44ef9d8d5e0d23e23c13,0951-7715-29-2-657}, 
the authors present an analysis of the long-time behaviour and ergodicity of the ensemble Kalman filter 
with arbitrary ensemble size establishing time uniform bounds to control the filter divergence with variance inflation techniques 
and ensuring in addition the existence of an invariant measure. 
Accuracy results have been recently established for a fixed ensemble size in the linear Gaussian setting, see \cite{doi:10.1002/cpa.21722,Tong2018} 
and for ensemble Kalman-Bucy filters applied to continuous-time filtering problems, 
see \cite{delmoral2018, doi:10.1137/17M1119056}.

For inverse problems, the large ensemble size limit has been investigated in \cite{ErnstEtAl2015}. It has been shown that the ensemble Kalman inversion \cls{(EKI)} is not consistent with the Bayesian
perspective in the nonlinear setting, but can be interpreted as a point
estimator of the unknown parameters. We will adopt this viewpoint throughout the paper and analyze the behavior of the EKI as an optimization method of the least-squares misfit functional. \cls{However, to motivate the algorithm, we will shortly introduce the Bayesian setting and derive the ensemble Kalman filter for inverse problems.} In \cite{SchSt2016}, it was demonstrated that the continuous time limit of the EKI algorithm is an interacting set of gradient flows, see also \cite{bergemann2010localization, bergemann2010mollified, Reich2011} for the continuous time limit of the EnKF in the data assimilation context. 
In the discrete setting, the connection to deterministic regularisation techniques is established 
in \cite{Iglesias2015, 2016InvPr..32b5002I}. \cls{} 
In the following, we will interpret the EKI method as a numerical discrete approximation of a stochastic differential equation, 
cp.\ \cite{bloemker2018strongly} and show well-posedness and asymptotic behaviour of the stochastic differential equation. 
Our work will extend the results from \cite{SchSt2016, doi:10.1080/00036811.2017.1386784} to the inversion with perturbed observations. \cls{Though both methods, i.e. the limit of the EKI with perturbed observations and the deterministic limit from  \cite{SchSt2016}, can be analysed from an optimization perspective, the EKI variant with perturbed observation is shown to be second order accurate, whereas the deterministic limit underestimates the covariance in the linear, Gaussian setting, see e.g. \cite{Evensen2003}. In addition, in the nonlinear setting, methods that add noise to data are reported to be more robust to assumptions about linearity and normality, see e.g. \cite{Dean2008} and the references therein. We therefore believe that the EKI with perturbed observations is a good starting point for methods (also of higher accuracy) in the nonlinear, non-Gaussian setting and that the analysis presented here provides valuable insights for the development of these methods.} 

Our contribution consists of providing a complete analysis of the ensemble Kalman inversion with perturbed observations for linear forward operators. 
The presented results hold true for arbitrary prior distributions on the unknown parameters, 
i.e.\ no Gaussian assumption is invoked. 
We want to stress that we analyze the algorithm in practical regimes by focusing on results for a fixed ensemble size. 
We study the continuous time limit of the ensemble Kalman inversion, which allows to establish well-posedness and accuracy results by exploiting the underlying structure of the limiting coupled stochastic differential equations for the particles. In particular, we make the following main contributions: 
\begin{itemize}
\item 
We prove the existence and uniqueness of solutions of the limiting system of stochastic differential equations, thus well-posedness of the algorithm.
\item
We quantify the ensemble collapse in the observation as well as in the parameter space. The ensemble collapse is characterized in terms of moments and almost sure convergence with given rate.
\item
In case of exact data, we establish convergence results to the truth using variance inflation. The convergence is characterized in terms of second moments and almost sure convergence with given rate. Under additional assumptions on the forward operator, 
the results in the data space can be transferred to the parameter space.  
\item
We provide numerical experiments which illustrate the theoretical results studied in this paper.
\end{itemize}
We do \emph{not} show strong convergence of the discrete EnKF iteration to continuous paths of the corresponding SDE. 
This would be interesting and there are preliminary results \cite{bloemker2018strongly}, but this is still ongoing research.

The remainder of the article is structured as follows. At the end of this section, we formulate the inverse problem and the Bayesian approach to it. 
Section \ref{sec2} is devoted to the ensemble Kalman inversion with perturbed observations. In section \ref{sec3} we formulate the continuous time limit of the algorithm, introduce the assumptions on the forward problem and prove the well-posedness of the method, 
i.e.\ we show the existence and uniqueness of strong solutions of the limit. Section \ref{sec4} presents the results on the ensemble collapse, in the data and parameter space. 
In section \ref{sec5}, we show convergence to the truth using variance inflation techniques. 
Numerical experiments illustrating the theoretical findings are given in section \ref{sec6}. 
Finally, in section \ref{sec:concl}, we conclude with a short summary of the main results and discussion of future work. 
In \ref{sec:app} auxiliary results are presented and and \ref{sec:app2} contains the proof on the higher-order ensemble collapse.

Let $\cG\in \mathcal C(\cX, \mathbb R^K)$ denote the forward response operator mapping the unknown parameters $u\in \cX$ to the data space $\mathbb R^K$, 
where $\cX$ is a separable Hilbert space and $K\in\mathbb N$ denotes the number of observations. 
We consider the inverse problem of recovering unknown parameters $u\in \cX$ from noisy observation $y\in \mathbb{R}^K$ given by
\begin{equation*}
y = \cG(u)+\eta\,,
\end{equation*}
where $\eta \sim\mathcal N(0,\Gamma)$ is a Gaussian with mean zero and covariance matrix $\Gamma$, 
which models the noise in the observations and in the model.

Following the Bayesian approach, for fixed $y\in \mathbb{R}^K$ 
we introduce the least-squares functional $\Phi(\cdot;y):\cX\to \mathbb R$ by
\begin{equation*}
\Phi(u ; y)=\frac{1}{2}|(y-\cG(u)) |_{\Gamma}^2\,.
\end{equation*}
with $|\cdot |_\Gamma:= |\Gamma^{-\frac12}\cdot |$ denoting the weighted Euclidean norm in $\mathbb R^K$. 
The unknown parameter $u$ is modeled as a $\cX$-valued random variable with prior distribution $\mu_0$. 
Thus, the pair $(u,y)$ is a jointly varying random variable on $\cX \times \mathbb R^K$. 
We assume for the observational noise that $\eta\sim\cN(0,\Gamma)$ is independent of $u\sim\mu_0$.

By Bayes' Theorem, the solution to the inverse problem is the $\cX$-valued random variable $u\mid y\sim\mu$ where the law $\mu$ is given by 
\begin{equation*}
\mu(du)=\frac{1}{Z}\exp(-\Phi(u ; y))\mu_0(du)
\end{equation*}
with the normalization constant $Z$, where
\begin{equation*}
Z:=\int_{\cX}\exp(-\Phi(u;y))\mu_0(du).
\end{equation*}
{Note that evaluation of the posterior requires evaluation of the forward model via $\Phi(u; y)$.}

\section{The EnKF for Inverse Problems}
\label{sec2}
{The Ensemble Kalman methodology consists of choosing an \emph{ensemble} of ``particles'' by drawing from the prior which are then transformed to a new set of particles via a linear Gaussian update. A good idea (see \cite{SchSt2016, 2016InvPr..32b5002I} for details) is to do this not in one big leap but in an iteration of steps. 
\cls{This amounts to interpolating the step from prior $\mu_0$ to the posterior $\mu$ by choosing an artificial time index $n$ and defining a sequence of measures $\mu_0, \mu_1, \mu_2, \ldots, \mu_N$ where $\mu_0$ is the prior and $\mu_N = \mu$ is the posterior, i.e.
\begin{equation*}
\mu_{n+1}(du)=\frac{1}{Z_n}\exp\bigl(-h\Phi(u;y)\bigr)\mu_n(du)\,, \quad n=0,\ldots, N-1
\end{equation*} 
with $h=N^{-1}$ and $Z_n=\int \exp(-h\Phi(u))\mu_n(du).$.}
The iterative form of the Ensemble Kalman methodology iterates the initial ensemble of particles through this set of intermediate measures. This is the setting we will constrain ourselves. Note again that any ``time'' $n$ or (later) $t$ is entirely artificial (transformation) time and independent of any physical time which may be present in the data.}

{From now on, the set $\{u_0^{(j)}\}_j$ denotes the initial ensemble of particles with each particle living in parameter space: $u_0^{(j)}\in \cX$. The iteration of particles consists of the set $\{u_n^{(j)}\}_{j, n}$ where $j$ is the ensemble index and $n$ the artificial time index. Each particle again is an element of parameter space: $u_n^{(j)}\in \cX$.}
\cls{
The measures $\mu_n$ will be approximated by an equally weighted sum of Dirac measures 
\begin{equation}
  {\mu_{n}}\simeq \frac 1J \sum_{j=1}^J \delta_{u_n}^{(j)}
  \end{equation}
  via the ensemble of particles. The initial ensemble is constructed based on the prior distribution and then mapped to the next iteration via a Gaussian approximation, i.e. given $\{u_n^{(j)}\}_{j, n}$, the transformed ensemble $\{u_{n+1}^{(j)}\}_{j, n}$ satisfies
  \begin{equation*}
\bar u_{n+1}=\overline u_{n}+K_n(y- \overline \mathcal G(u_n)) \qquad C(u_{n+1})=C(u_n)-K_nC^{pu}(u_n)
\end{equation*}
with $K_n=C^{up}(u_{n})(C^{pp}(u_{n})+\frac 1h\Gamma)^{-1}$.
The operators $C^{pp}$, $C^{up}$ and $C^{pu}$ are the empirical covariances {defined on $\cX^J$ } by
\begin{eqnarray*}
C^{pp}(u)&=\frac{1}{J}\sum\limits_{j=1}^J(\cG(u^{(j)})-\overline{\cG})\otimes(\cG(u^{(j)})-\overline{\cG}),\\
C^{up}(u)&=\frac{1}{J}\sum\limits_{j=1}^J(u^{(j)}-\overline{u})\otimes(\cG(u^{(j)})-\overline{\cG}),\\
C^{pu}(u)&=\frac{1}{J}\sum\limits_{j=1}^J(\cG(u^{(j)})-\overline{\cG})\otimes (u^{(j)}-\overline{u})
\end{eqnarray*}
{where {$u$ is short for the multiindex vector $(u^{(j)})_j\in \cX^J$} and $\otimes$ denotes the tensor product (or rank one operator) given by 
\begin{equation*}
z_1\otimes z_2:\cH_2\to\cH_1\ \mbox{with}\ h\mapsto z_1\otimes z_2(h):=\langle z_2,h\rangle_{\cH_2}\cdot z_1
\end{equation*}
for Hilbert spaces $(\cH_1,\langle\cdot,\cdot\rangle_{\cH_1}),(\cH_2,\langle\cdot,\cdot\rangle_{\cH_2})$ and $z_1\in\cH_1, z_2\in\cH_2$.} The empirical means are given by
\begin{equation*}
\overline{u}=\frac{1}{J}\sum\limits_{j=1}^Ju^{(j)}, \qquad
\overline{\cG}=\frac{1}{J}\sum\limits_{j=1}^J\cG(u^{(j)}).
\end{equation*}
The transformation of the ensemble from iteration $n$ to $n+1$ is not uniquely determined via the Kalman update formula. For the EKI with perturbed observations, the update formula is shown to be satisfied in the mean, see e.g. \cite{Evensen2003}. 
}

\sw{ Although we consider a Gaussian approximation for the measures $\mu_n$, for our theoretical results we do not require any assumption on Gaussian prior distributions.
}

For a given artificial step-size $h>0$ and $J\ge2$ particles, the EnKF iteration for the $j$-th particle is given by
\begin{equation}
\label{Iteration}
\fl u_{n+1}^{(j)} = u_n^{(j)}+C^{up}(u_n)(C^{pp}(u_n)+h^{-1}\Gamma)^{-1}(y_{n+1}^{(j)}-\cG(u_n^{(j)})),\quad j = 1,\dots,J\,,
\end{equation}
where the initial particles $u_0^{(j)}$, $j=1,\ldots, J$ are draws from the prior distribution. In each step, we consider artificially perturbed data
\begin{equation*}
y_{n+1}^{(j)} = y+\xi_{n+1}^{(j)}\;,
\end{equation*}
where the perturbations $\xi_{n+1}^{(j)}$, with respect to both $j$ and $n$, are i.i.d.\ 
random variables distributed according to $\cN(0,h^{-1}\Gamma)$.
For a derivation of the EnKF for inverse problems, we refer to \cite{StLawIg2013}.
\section{Continuous Time Limit}\label{Cont_time}
\label{sec3}

The continuous time limit of the {discrete } EnKF inversion {\eref{Iteration}} is formally a time discretization of the following SDE:
\begin{equation}\label{cont_v1}
\d u_t^{(j)} = C^{up}(u_t)\Gamma^{-1}(y-\cG(u_t^{(j)}))\,\d t+C^{up}(u_t)\Gamma^{-1/2} \, \d W_t^{(j)}.
\end{equation}
Using the definition of the empirical covariance, \eref{cont_v1} can be formulated equivalently as
\begin{equation}\label{cont_v2}
\d u_t^{(j)} = \frac{1}{J}\sum\limits_{k=1}^J\left\langle\cG(u_t^{(k)})-\overline{\cG}_t,(y-\cG(u_t^{(j)}))\,\d t+\sqrt{\Gamma} \, \d W_t^{(j)}\right\rangle_{\Gamma}(u_t^{(k)}-\overline{u}_t)
\end{equation}
with $\langle\cdot,\cdot\rangle_{\Gamma}=\langle\Gamma^{-\frac{1}{2}}\cdot,\Gamma^{-\frac{1}{2}}\cdot\rangle$, 
where $\langle\cdot,\cdot\rangle$ denotes the standard Euclidean inner-product on $\mathbb R^K$. 
The processes $W^{(j)}$ are independent Brownian motions on $\mathbb R^K$. 
We further denote by $\mathcal F_t=\sigma(u_s, s\le t)$ the filtration introduced by the particle dynamics.
{Most of the time we will write $u_t$ or $u(t)$ to emphasize the dependence on time $t$. If the time dependence is clear from the context, we simplify the notation to $u$.} 

The formulation \eref{cont_v2} reveals 
that solutions  satisfy a generalization of the subspace property
of \cite[Theorem 2.1]{StLawIg2013} to continuous time.
\begin{lemma}\label{lin_span}
Assume that $\cG$ is locally Lipschitz and 
let $\cS$ {be} the linear span of $\{u_0^{(j)}\}_{j=1}^J$, then $u_t^{(j)}\in\cS$ for all $(t,j)\in [0,\infty)\times\{1,\dots, J\}$ {almost surely}.
\end{lemma}
We do not give all the technical details of the proof. First one needs to show that the initial value problem related to (\ref{cont_v2}) 
has a unique $\cX$-valued solution, which is assured by a local Lipschitz-property of the  drift and diffusion. 
As the vector field on the right hand side of (\ref{cont_v2}) maps $\cS$ in $\cS$, we can show by the same argument 
that there is also a unique $\cS$-valued solution.
Thus by the uniqueness in any $\cX$ both solutions coincide, 
and all solution must be  $\cS$-valued.
\cls{The subspace property reveals the regularization effect of the ensemble of particles in the inverse setting. Due to Lemma \ref{lin_span}, the EKI estimate lies in the subspace spanned by the initial ensemble, which is usually a much smaller space than the original parameter $\cX$. Thus, the discretization via the ensemble of particles can be interpreted as a regularization or stabilization of the inverse problem.}

\subsection{The Linear Problem}

For the whole paper, we will assume that the forward response operator is linear, i.e.\ $\cG(\cdot)=A\cdot$ with $A\in\cL(\cX,\R^K)$.
Then the continuous time limit (\ref{cont_v2}) reads as
\begin{equation}\label{lin_v1}
\d u_t^{(j)} = \frac{1}{J}\sum\limits_{k=1}^J\left\langle A(u_t^{(k)}-\overline{u}_t),(y-Au_t^{(j)})\,\d t+\sqrt{\Gamma} \, \d W_t^{(j)}\right\rangle_{\Gamma}(u_t^{(k)}-\overline{u}_t).
\end{equation}
We simplify notation by defining the empirical covariance operator 
\begin{equation} \label{eq:empCov}
C(u)=\frac1J\sum\limits_{k=1}^J(u^{(k)}-\overline{u})\otimes(u^{(k)}-\overline{u}).
\end{equation}
Thus equation (\ref{lin_v1}) can be rewritten in the form
\begin{equation}\label{lin_v2}
\d u_t^{(j)} = C(u_t)A^*\Gamma^{-1}(y-Au_t^{(j)})\,\d t+C(u_t)A^*\Gamma^{-1/2}\,\d W_t^{(j)}.
\end{equation}

\subsection{Well-posedness of the EnKF inversion}\label{sec_wellposed}

This section is devoted to proving existence and uniqueness of global solutions of the set of coupled SDEs (\ref{lin_v2}). 
Again, the local existence and uniqueness of $\cX$-valued local solutions to (\ref{lin_v2}) is straightforward by the local Lipschitz-property of the drift and diffusion on the right-hand side.
Thus we  rely on the subspace property of Lemma \ref{lin_span}, 
and first show that we can reduce the $\cX$-valued setting without loss of generality to a finite-dimensional setting.

\begin{lemma}
Without loss of generality we assume that the initial ensemble $(u_0^{(j)})_{j\in\{1,\dots,J\}}$ is linearly independent almost surely
and spans a $J$-dimensional vector space $\cS$. 

Then there exists a linear operator $\tilde A:\R^J\to\R^K$ such that equation \eref{lin_v2} restricted to $\cS$ is equivalent to 
\begin{equation}
\label{SDE_finite}
\d v_t^{(j)} = \frac1J\sum\limits_{k=1}^J\left\langle \tilde Av_t^{(k)}-\tilde A\overline v_t,(y-\tilde Av_t^{(j)})\,\d t+\Gamma^{\frac12}\,\d W_t^{(j)}\right\rangle_\Gamma(v_t^{(k)}-\overline v_t)
\end{equation}
for {$v_t^{(j)}\in\R^J$, $\overline v_t:=\frac1J\sum\limits_{k=1}^Jv^{(k)}_t$}, in the following sense: 
For $u_t^{(j)} = \sum\limits_{k=1}^J(v_t^{(j)})_k\cdot u_0^{(k)}$ one has that $u_t$ is a $\cS$-valued solution 
of \eref{lin_v2} if and only if $v_t$ is a solution of (\ref{SDE_finite}).
\end{lemma}
\begin{proof}
By Lemma \ref{lin_span}, any $\cS$-valued process $u(t)$ can be uniquely expanded as a linear combination $u^{(j)}(t)=\sum\limits_{l=1}^Jv^{(j)}_l(t)\cdot u^{(l)}(0)$ for every $j\in\{1,\dots,J\}$, $t\ge0$ and coordinates $v^{(j)}_l(t)\in\R$. 
Let $\Phi^{-1}:\R^J\to\cS$ denote the basis isomorphism, i.e.\ {$\Phi:\cS\to\R^J$ with $u=\sum\limits_{l=1}^Jv_lu^{(l)}(0)\stackrel{\Phi}{\mapsto} (v_1,\dots,v_J)^\top$. }
Since $\Phi$ is a linear isomorphism, {\eref{lin_v2} can be equivalently transformed to}
\begin{eqnarray*}
\fl \d\Phi(u^{(j)}(t))
=\Phi(\d u^{(j)}(t)) \\
\fl=\frac1J\sum\limits_{k=1}^J\langle A(u^{(k)}(t)-\overline{u}(t)),(y-Au^{(j)}(t))\,\d t+\Gamma^{\frac12}\,\d W_t^{(j)}\rangle_\Gamma(\Phi(u^{(k)}(t))-\Phi(\overline{u}(t)))\,.
\end{eqnarray*}
{Thus, with $\tilde A = A\Phi^{-1}$, we obtain}
\begin{eqnarray*}
\fl\d\Phi(u^{(j)}(t)) 
&=& \frac1J\sum\limits_{k=1}^J\langle \tilde A\Phi(u^{(k)}(t)-\overline{u}(t)),(y-\tilde A\Phi(u^{(j)}(t)))\,\d t+\Gamma^{\frac12}\,\d W_t^{(j)}\rangle_\Gamma\\ 
&&\qquad\cdot(\Phi(u^{(k)}(t))-\Phi(\overline{u}(t)))
\end{eqnarray*}
{The assertion follows with $v^{(j)}:=\Phi(u^{(j)})$.}
\end{proof}
\begin{remark}
By the previous lemma solving equation (\ref{lin_v2}) is equivalent to solving the finite dimensional equation (\ref{SDE_finite}). Thus, to simplify notation we will assume without loss of generality that $\cX=\R^I,\ I\in\N,\ I\le J$. In the case of linearly independent initial ensemble we can assume $I=J$.

For the study of the dynamical behavior of the ensemble, we will {sometimes} require the following assumption for results in the parameters space:
\begin{equation}\label{A1}
\mbox{The linear operator $\tilde A$ defined above is one-to-one.}
\end{equation}
\cls{Note that Assumption \eref{A1} seems to be {a rather strict assumption: It requires that the forward operator  ``sees everything'' and secondly, this means that} $(\tilde A\Phi(u_0^{(j)}))_{j\in\{1,\dots,J\}}\in\R^K$ is linearly independent. 
This implies the restriction on the number of particles $J\le K$.
However, note that this assumption is on the operator $\tilde A$, i.e. we do not assume that $A$ is one-to-one. The discretization of the parameter space via the ensemble of particles acts as a regularization of the inverse problem in this setting.  
We will need assumption \eref{A1} only when we want to prove dynamical properties in \textit{parameter space}. 
This makes sense as we cannot hope for convergence to the true parameter if the forward operator is indifferent with respect to some components of this parameter value. 
Our convergence results in the \textit{observation} space hold without assumption \eref{A1}.}
\end{remark}

In order to prove the existence and uniqueness of global solutions
we rewrite the set of coupled SDEs (\ref{lin_v2}) as a single SDE of the following form:
\begin{equation*}
du_t = F(u_t)\,dt+G(u_t)\, dW_t,
\end{equation*}
with $u_t = (u_t^{(j)})_{j\in\{1,\dots,J\}}\in\R^{IJ\times1}, W_t = (W_t^{(j)})_{j\in\{1,\dots,J\}}\in\R^{J^2\times1}$ and
\begin{eqnarray*}
F(x) &= (C(x)A^* \Gamma^{-1}(y-Ax^{(j)}))_{j\in\{1,\dots,J\}}\in\R^{IJ\times1},\\
G(x) &= \diag(C(x)A^* \Gamma^{-\frac12})_{j\in\{1,\dots,J\}}\in\R^{IJ\times J^2},
\end{eqnarray*}
where $x=(x^{(j)})_{j\in\{1,\dots,J\}}\in\R^{IJ\times1}$ and $\diag(B_j )_{j\in\{1,\dots,J\}}$ is a diagonal block matrix with matrices $(B_j)_{j\in\{1,\dots,J\}}$ on the diagonal.
For a given matrix $B=(b_{ij})_{ij}\in\mathbb R^{n\times m}, \ m,n\in\mathbb N$, the Frobenius norm $\|B\|_F$ is defined by $\|B\|_F^2=\trace{B^\top  B}=\sum_{i,j}b_{ij}^2\ge \|B\|_2^2$.

We will now formulate and prove the main result of this section on the well-posedness of the EnKF inversion.
\begin{theorem}\label{Thm_Main1}
Let 
$
u_0=(u_0^{(j)})_{j\in\{1,\dots,J\}}
$
be $\cF_0$-measurable maps $u_0^{(j)}:\Omega\to\cX$ {which are linearly independent almost surely}. Then for all $T\ge0$ there exists a unique strong solution $(u_t)_{t\in[0,T]}$ (up to $\PP$-indistinguishability) of the set of coupled SDEs \eref{lin_v2}.
\end{theorem}

\begin{proof}
For the proof we will assume without loss of generality that $\cX=\mathbb{R}^I$ for $I$ sufficiently large, as discussed before.
The proof of existence and uniqueness of local strong solutions for \eref{lin_v2} (up to a stopping-time) is standard, 
due to the local Lipschitz property of the drift $F$ and the diffusion $G$.
Note that both are polynomials.

The global existence of a strong solution is based on stochastic Lyapunov theory. See for example Theorem 4.1 of \cite{Hasminski}.
We only need to construct a function $V\in C^2(\cX;\R_+)$ such that for some constant $c>0$
\begin{equation}
 LV(x):= \nabla V(x) \cdot F(x)+\frac12\trace(G^T(x) \mathrm{Hess}[V](x)G(x))\le cV(x)  \label{e:A}
\end{equation}
and
\begin{equation}
 \inf_{|x|>R}V(x)\to\infty\ \mbox{as $R\to\infty$} \label{e:B}
\end{equation}
hold true. 
 
We can {uniquely} decompose $y\in\R^K$ as $y=y_1+y_2$, with $y_1\in\cR(\Gamma^{-\frac12}A)$ and $y_2\in\cR(\Gamma^{-\frac12}A)^\perp$, where $\cR(\Gamma^{-\frac12}A)$ denotes the image of $\Gamma^{-\frac12}A$. We fix $\tilde u\in\R^J$ such that $\Gamma^{-\frac12}A\tilde u=y_1$ and define the Lyapunov function
\begin{equation*}
V(u):= V_1(u) + V_2(u) = \frac1J\sum\limits_{j=1}^J\|u^{(j)}-\overline u\|^2+\|\overline u-\tilde u\|^2.
\end{equation*}
Obviously, (\ref{e:B}) is satisfied.

The generator $L$ applied to $V$ is given by $LV=LV_1+LV_2$ with
\begin{eqnarray*}
LV_1(u) 	&=& -\frac{J+1}{J^3}\sum\limits_{j, l=1}^J\langle u^{(j)}-\overline u,u^{(l)}-\overline u\rangle\langle\Gamma^{-\frac12}A(u^{(l)}-\overline u),\Gamma^{-\frac12}A(u^{(j)}-\overline u)\rangle\\
LV_2(u) 	
		&=&-\frac2J\sum\limits_{l=1}^J\langle \overline u-\tilde u,u^{(l)}-\overline u\rangle\langle\Gamma^{-\frac12}A(u^{(l)}-\overline u),\Gamma^{-\frac12}A(\overline u-\tilde u)\rangle\\
		&&\quad+\frac1{J^3}\sum\limits_{j, l=1}^J\langle u^{(j)}-\overline u,u^{(l)}-\overline u\rangle\langle\Gamma^{-\frac12}A(u^{(l)}-\overline u),\Gamma^{-\frac12}A(u^{(j)}-\overline u)\rangle,
\end{eqnarray*}
where we used $\langle\Gamma^{-\frac12}A(u^{(l)}-\overline u),y_2\rangle = 0$ for all $l\in\{1,\dots,J\}$ wich is true by construction. 
Thus, as $A^\star\Gamma^{-1}A$ is a symmetric non-negative matrix by Lemma \ref{nonneg} the nonnegativity of the generator follows.
\begin{eqnarray*}
 LV(u) 
		&=& -\frac2J\sum\limits_{l=1}^J\langle \overline u-\tilde u,u^{(l)}-\overline u\rangle\langle\Gamma^{-\frac12}A(u^{(l)}-\overline u),\Gamma^{-\frac12}A(\overline u-\tilde u)\rangle\\
		&&-\frac1{J^2}\sum\limits_{j, l=1}^J\langle u^{(j)}-\overline u,u^{(l)}-\overline u\rangle\langle\Gamma^{-\frac12}A(u^{(l)}-\overline u),\Gamma^{-\frac12}A(u^{(j)}-\overline u)\rangle\\
		&\le& 0.
\end{eqnarray*}
Thus (\ref{e:A}) holds true, for all $c>0$.
\end{proof}

\section{Quantification of the Ensemble Collapse}\label{Asymp_beh}
\label{sec4}
The dynamics of the Ensemble Kalman filter as presented here can be decomposed into two parts: 
{
\begin{itemize}
\item Ensemble collapse. This means convergence of all ensemble members to their joint mean (``The estimator becomes more confident'').
\item Convergence of the ensemble mean. {This means that} the ensemble mean will tend to a parameter value which is consistent with the data. 
\end{itemize}
Those two notions are totally different in concept but also strongly intertwined in the dynamics of the EnKF (see also the discussion at the beginning
of section \ref{sec5}).
}

We start by quantifying the ensemble collapse.
We will present results in the data (or observation) space as well as in the parameter space.

{For the further analysis, we introduce the centered quantities
\begin{equation*}
e^{(j)} = u^{(j)}-\overline{u}, \qquad r^{(j)} = u^{(j)}-u^{\dagger},
\end{equation*}
where $e^{(j)}$ denotes difference of each particles to the mean and $r^{(j)}$ denotes the residuals. Here, the data $y$ is the perturbed image of a truth $u^\dagger\in\cX$ under $A$, i.e.\ $y = Au^\dagger+\eta$. The quantities satisfy} the following equations {(note that $C(u) = C(e)$ as the mean of the $e^{(j)}$ vanishes)}
\begin{eqnarray}
\d e_t^{(j)} 	&= -C({e_t})A^*\Gamma^{-1}Ae_t^{(j)}\,\d t+C({e_t})A^*\Gamma^{-\frac12}\,\d (W_t^{(j)}-\overline W_t),\label{eq:sde_e}\\
\d r_t^{(j)}	&= \d u_t^{(j)} = C(u_t)A^*\Gamma^{-1}(y-Au_t^{(j)})\,\d t+C(u_t)A^*\Gamma^{-1/2}\,\d W_t^{(j)},
\end{eqnarray}
with $\overline W_t:=\frac1J\sum\limits_{j=1}^JW^{(j)}$.
{ The dynamical behavior of the empirical mean is given by} 
\begin{equation*}
\d\overline u_t=\frac1J\sum\limits_{k=1}^J(u_t^{(k)}-\overline u_t)\langle A(u_t^{(k)}-\overline u_t),(y-A\overline u_t)\,\d t+\Gamma^{\frac12}\,\d\overline W_t\rangle_{\Gamma}.
\end{equation*}
 To simplify notation, we also introduce the transformed quantities
\begin{equation*}
\mathfrak r^{(j)}:=\Gamma^{-\frac12}Ar^{(j)},\qquad \mathfrak e^{(j)}:=\Gamma^{-\frac12}Ae^{(j)}=\mathfrak r^{(j)}-\overline{\mathfrak r}
\end{equation*}
denoting the residuals {in observation space} and the mapped difference of each particle to the empirical mean.

{We will now make a first step towards proving ensemble collapse. As we work with SDEs, any dynamical property can only hold in some probabilistic sense. The following lemma shows that we have ensemble collapse in the $L^p$ sense, with the upper bound for valid parameters $p$ being dependent on the number of particles $J$.}
{\begin{lemma}\label{Monotonie}
Let $p\in[2,J+3)$ and
$
u_0=(u_0^{(j)})_{j\in\{1,\dots,J\}}
$
be $\cF_0$-measurable maps $u_0^{(j)}:\Omega\to\cX$ such that 
$
\E[\frac1J\sum\limits_{j=1}^J|\mathfrak e_0^{(j)}|^p]<\infty. 
$
Then \begin{equation*}t\in[0,\infty)\mapsto\|\mathfrak e_t\|_{\cL_p(\Omega,\R^K)}:=\E\left[\frac1J\sum\limits_{j=1}^J|\mathfrak e_t^{(j)}|^p\right]^{\frac1p}\end{equation*} is monotonically decreasing in $t$. Furthermore there exists a constant $C>0$ such that for all $t\ge0$ \begin{equation*}\int_0^t\E\left[\frac1J\sum\limits_{j=1}^J|\mathfrak e_s^{(j)}|^{p+2}\right]\,\d s<C.\end{equation*}
\end{lemma}

\begin{proof}
We will prove the assertion in the case $p=2$, in order to give the key ideas. 
The case $p>2$ is very similar, but much more  technical. 
We postpone all details in that case to the appendix.

Applying $\Gamma^{-\frac12}A$ to $e^{(j)}$ implies that the quantity $\mathfrak e^{(j)}$ satisfies {(see \eref{eq:sde_e} and \eref{eq:empCov})}
\begin{eqnarray*}
\d \mathfrak e_t^{(j)} &=-C(\mathfrak e_t)\mathfrak e_t^{(j)}\,\d  t+C(\mathfrak e_t)\,\d (W_t^{(j)}-\overline{W}_t)\\
			&=-\frac1J\sum\limits_{k=1}^J\mathfrak e_t^{(k)}\langle\mathfrak e_t^{(k)},\mathfrak e_t^{(j)}\rangle\,\d  t+\frac1J\sum\limits_{k=1}^J\mathfrak e_t^{(k)}\langle \mathfrak e_t^{(k)},\d (W_t^{(j)}-\overline{W}_t)\rangle.
\end{eqnarray*}
It\^{o}'s formula gives
\begin{eqnarray*}
\d |\mathfrak e_t^{(j)}|^2	&=2\langle\mathfrak e_t^{(j)},\d \mathfrak e_t^{(j)}\rangle+\langle\d \mathfrak e_t^{(j)},\d \mathfrak e_t^{(j)}\rangle\\
			&=-\frac2J\sum\limits_{k=1}^J\left\langle\mathfrak e_t^{(j)},\mathfrak e_t^{(k)}\right\rangle^2\,\d  t+2\mathfrak e_t	^{(j)T}C(\mathfrak e_t)\,\d (W_t^{(j)}-\overline{W}_t)\\
			&\quad+\frac{1}{J^2}\sum\limits_{k,l=1}^J\left\langle\mathfrak e_t^{(k)},\mathfrak e_t^{(l)}\right\rangle\left\langle\mathfrak e_t^{(k)},\d (W_t^{(j)}-\overline{W})\right\rangle\left\langle\mathfrak e_t^{(l)},\d (W_t^{(j)}-\overline{W})\right\rangle
\end{eqnarray*}
and with Lemma \ref{BM_innerproduct} to evaluate the It\^o correction we get
\begin{equation*}
\fl \d |\mathfrak e_t^{(j)}|^2 =-\frac2J\sum\limits_{k=1}^J\langle\mathfrak e_t^{(j)},\mathfrak e_t^{(k)}\rangle^2\,\d  t+2\mathfrak e_t^{(j)T}C(\mathfrak e_t)\,\d (W_t^{(j)}-\overline{W}_t)+\frac{J-1}{J^3}\sum\limits_{k,l=1}^J\langle\mathfrak e_t^{(k)},\mathfrak e_t^{(l)}\rangle^2\,\d  t\,.
\end{equation*}
Summing over all particles leads to
\begin{eqnarray*}
\fl\d \left(\frac1J\sum\limits_{j=1}^J|\mathfrak e_t^{(j)}|^2\right) 	
									&=& -\frac{J+1}{J^3}\sum\limits_{j,k=1}^J\left\langle\mathfrak e_t^{(j)},\mathfrak e_t^{(k)}\right\rangle^2\,\d  t+\frac2J\sum\limits_{j=1}^J\mathfrak e_t^{(j)\top}C(\mathfrak e_t)\,\d (W_t^{(j)}-\overline{W}_t)\\
									&=& -\frac{J+1}{J^3}\sum\limits_{j,k=1}^J\left\langle\mathfrak e_t^{(j)},\mathfrak e_t^{(k)}\right\rangle^2\,\d  t+\frac2J\sum\limits_{j=1}^J\mathfrak e_t^{(j)\top}C(\mathfrak e_t)\,\d W_t^{(j)}.
\end{eqnarray*}
The last step follows from $\sum_j \mathfrak e^{(j)} = 0$. This yields
\begin{equation}\label{eq:IntSum_e}
\fl\eqalign{
\frac1J\sum\limits_{j=1}^J|\mathfrak e_t^{(j)}|^2 &- \frac1J\sum\limits_{j=1}^J|\mathfrak e_0^{(j)}|^2 \cr
&= -\frac{J+1}{J^3}\int_0^t\sum\limits_{j,k=1}^J\langle\mathfrak e_t^{(j)},\mathfrak e_t^{(k)}\rangle^2\,\d  t +\frac2J\int_0^t\sum\limits_{j=1}^J\mathfrak e_t^{(j)\top}C(\mathfrak e_t)\,\d W_t^{(j)}.}
\end{equation}
Now we cannot simply take the expectation, as we do not know that the stochastic integral is a martingale. We need a localization.
Set $t,s\ge0$ and let  $(\tau_n)_{n\in\N}$ with $\tau_n\stackrel{n}{\to}\infty$ a.s.\  be a sequence of deterministically bounded stopping times, such that 
\begin{equation*}
\int_s^{s+(t\wedge\tau_n)}\mathfrak e_s^{(j)T}C(\mathfrak e_s)\,\d  W_s^{(j)}
\end{equation*} 
is a martingale for every $j\in\{1,\cdots,J\}$. {This is possible by definition of local martingales, with any stochastic integral being one.}
For example we can take for $\tau_n$ the minimum of $n$ and the first exit time of $\mathfrak{e}_s$ at radius $n$. 
Then, for all $n\in\N$, from \eref{eq:IntSum_e} (after rebasing the integration interval from $[0,t]$ to $[s, s+t]$) we obtain
\begin{equation*}
\fl\E\left[\frac1J\sum\limits_{j=1}^J|\mathfrak e_{s+(t\wedge\tau_n)}^{(j)}|^2\right] -\E\left[\frac1J\sum\limits_{j=1}^J|\mathfrak e_s^{(j)}|^2\right]= -\E\left[\int_s^{s+(t\wedge\tau_n)}\frac{J+1}{J^3}\sum\limits_{j,k=1}^J\langle \mathfrak e_r^{(j)},\mathfrak e_r^{(k)}\rangle^2\,\d  r\right]
\end{equation*}
{As} $\tau_n\to\infty$,
applying Fatou's lemma on the left hand side and applying the monotone convergence theorem on the right hand side gives

\begin{equation}\label{Inequality1}
\fl\E\left[\frac1J\sum\limits_{j=1}^J|\mathfrak{e}_{s+t}^{(j)}|^2\right] - \E\left[\frac1J\sum\limits_{j=1}^J|\mathfrak{e}_s^{(j)}|^2\right] \le -\E\left[\int_s^{s+t}\frac{J+1}{J^3}\sum\limits_{j,k=1}^J\langle \mathfrak e_r^{(j)},\mathfrak e_r^{(k)}\rangle^2\,\d  r\right]\le0,
\end{equation}
which implies that $\E[\frac1J\sum\limits_{j=1}^J|\mathfrak{e}_t^{(j)}|^2]$ is monotonically decreasing in $t$. 

Finally,
\begin{equation*}
\fl\int_0^t\E\left[\frac{J+1}{J^3}\sum\limits_{j=1}^J|\mathfrak{e}_s^{(j)}|^4\right]\,\d s
\le\int_0^t\E\left[\frac{J+1}{J^3}\sum\limits_{j,k=1}^J\langle \mathfrak{e}_s^{(j)},\mathfrak{e}_s^{(k)}\rangle^2\right]\,\d s\le\E\left[\frac1J\sum\limits_{j=1}^J|\mathfrak{e}_0^{(j)}|^2\right],
\end{equation*}
{where the first inequality is trivial by inserting non-negative terms in the sum and the second inequality is \eref{Inequality1} with $s=0$. This} proves the second claim.

Let us finally remark that  $\tau_n\to\infty$ necessarily holds. If we assume that  $\tau_n\to\tau_*$ 
then the previous argument with $s=0$ and arbitrary $T>0$ 
gives $\E[\frac1J\sum\limits_{j=1}^J|\mathfrak{e}_{t\wedge\tau_*}^{(j)}|^2] <\infty$.
Thus $t<\tau_*$ for our choice of stopping time.

\end{proof}

The main obstacle in quantifying the ensemble collapse is {proving that the stochastic integral in \eref{eq:IntSum_e} is actually a true martingale.} 
See Lemma \ref{martingale}  for details. 

\begin{theorem}\label{Thm_main2}
Let
$
u_0=(u_0^{(j)})_{j\in\{1,\dots,J\}}
$
be $\cF_0$-measurable random variables $u_0^{(j)}:\Omega\to\cX$ such that 
$
C_0:=\E[\frac1J\sum\limits_{j=1}^J|\mathfrak e_0^{(j)}|^2]<\infty. 
$
Then, the ensemble collapse is quantified by
\begin{equation}\label{eq:enscoll}
\E\left[\frac1J\sum\limits_{j=1}^J|\mathfrak e_t^{(j)}|^2\right]
\le\frac{1}{\frac{J+1}{J^2}t+\frac1{C_0}}.
\end{equation}
\end{theorem}
\begin{proof}
By  Lemma \ref{martingale} we can directly take expectations in \eref{eq:IntSum_e} to obtain 
\begin{equation*}
\E\left[\frac1J\sum\limits_{j=1}^J|\mathfrak e_t^{(j)}|^2\right] 
= \E\left[\frac1J\sum\limits_{j=1}^J|\mathfrak e_0^{(j)}|^2\right] - \frac{J+1}{J^3}\int_0^t\E\left[\sum\limits_{j,k=1}^J\langle \mathfrak{e}_s^{(j)},\mathfrak{e}_s^{(k)}\rangle^2\right]\,\d s.
\end{equation*}

Note  that {by dropping the non-negative mixed terms $j\neq k$ and by using Jensen's and Young's inequality} 
\begin{equation*}
\frac{J+1}{J^3}\E\left[\sum\limits_{j,k=1}^J\langle \mathfrak e_s^{(j)},\mathfrak{e}_s^{(k)}\rangle^2\right]
\ge \frac{J+1}{J^2}\E\left[\frac1J\sum\limits_{j=1}^J|\mathfrak e_s^{(j)}|^2\right]^2.
\end{equation*} 
Thus setting $t\mapsto h(t):=\E[\frac1
J\sum\limits_{j=1}^J|\mathfrak e_t^{(j)}|^2]$ 
we can write 
\begin{equation*}
h(t) = h(0) - \frac{J+1}{J^2}\int_0^th^2(s)\,\d s - \int_0^t p(s)ds
\end{equation*}
for a non-negative function $p\geq0$.
Hence, we can differentiate to obtain the differential inequality 
\[
h' \leq - \frac{J+1}{J^2} h^2\;,
\]
from which by a comparison argument for scalar ODE it follows that 
\begin{equation*}
h(t)=\E\left[\frac1J\sum\limits_{j=1}^J|\mathfrak e_t^{(j)}|^2\right]\le\frac1{\frac{J+1}{J^2}t+\frac1{h(0)}}\;.
\end{equation*} 
\end{proof}

\begin{corollary}
Under the same assumptions as in Theorem \ref{Thm_main2} and under Assumption \eref{A1} it holds true that
\begin{equation*}
\E\left[\frac1J\sum\limits_{j=1}^J|e_t^{(j)}|^2\right]\le\frac1{\sigma_{\min}}\frac{1}{\frac{J+1}{J^2}t+\frac1{C_0}},
\end{equation*}
where $\sigma_{\min}$ is the smallest eigenvalue of the positive definite operator $A^*\Gamma^{-1}A$.
\end{corollary}
\begin{proof}
The assertion follows directly from the inequality
\begin{equation*}
|\mathfrak e^{(j)}|^2 = |\Gamma^{-\frac12}Ae^{(j)}|^2 = \langle e^{(j)},A^*\Gamma^{-1}Ae^{(j)}\rangle\ge \sigma_{\min}|e^{(j)}|^2,
\end{equation*}
since $A^*\Gamma^{-1}A$ is positive definite.
\end{proof}
\begin{remark}
Note that the bound in \eref{eq:enscoll} deteriorates with growing number of particles $J$, i.e.\ the result does not quantify the ensemble collapse in the large ensemble size limit. However, the presented analysis is tailored for fixed ensemble size and we will demonstrate in the numerical experiments that the derived bound \eref{eq:enscoll} can be efficiently used to quantify the collapse in this setting.
\end{remark}

\subsection{Higher-order ensemble collapse}

Here we state the result for higher moments and postpone the proof to the appendix, as they are very similar to but technically more involved than the case $p=2$. 
\begin{theorem}\label{Thm_main3}
Let $p\in(2,\frac{J+3}2)$ and let 
$
u_0=(u_0^{(j)})_{j\in\{1,\dots,J\}}
$
be $\cF_0$-measurable maps $u_0^{(j)}:\Omega\to\cX$ such that 
$
\E[\frac1J\sum\limits_{j=1}^J|\mathfrak e_0^{(j)}|^p]<\infty. 
$
Then it holds true that
\begin{equation*}
\E\left[\frac1J\sum\limits_{j=1}^J|\mathfrak e_t^{(j)}|^p\right]\le \frac{J^{\frac p2}}{\left(\frac2pC(p,J)K^{-\frac2p}J^{1-\frac2p}t+\left(K^{\frac{p-1}2}\E\left[\frac1J\sum\limits_{j=1}^J|\mathfrak e_0^{(j)}|^p\right]\right)^{-\frac2p}\right)^{\frac p2}}
\end{equation*}
with $C(p,J):=\frac{p}{J^2} (1 - \frac{(p-2+J)\cdot(J-1)}{2J^2}-\frac{p-2}{2J^2})$.
\end{theorem}
\begin{proof}
The proof based on It\^o's formula and a comparison principle for ODEs is very similar to the case $p=2$. 
Details can be found in the appendix. 

\end{proof}
{
\begin{remark}
Note that a larger ensemble seems to \textit{regularize} the dynamics.
The higher the ensemble number $J$, the larger is the highest moment of ensemble collapse we can bound. 

The restriction $2p<J+3$ comes from the fact that we need the martingale property of the stochastic integral, which we obtain from the bounds in Lemma \ref{Monotonie}.
\end{remark}}
\begin{corollary}\label{cor1}
Under the same assumptions as in Theorem \ref{Thm_main3} and under Assumption \eref{A1} it holds true that
\begin{equation*}
\fl\E\left[\frac1J\sum\limits_{j=1}^J |e_t^{(j)}|^p\right]\le\frac{J^{\frac p2}}{\left(\sigma_{\min}\cdot\frac2pC(p,J)K^{-\frac2p}J^{1-\frac2p}t+\left(K^{\frac{p-1}2}\E\left[\frac1J\sum\limits_{j=1}^J|\mathfrak e_0^{(j)}|^p\right]\right)^{-\frac2p}\right)^{\frac p2}},
\end{equation*}
where $\sigma_{\min}$ is the smallest eigenvalue of the positive definite operator $A^*\Gamma^{-1}A$ and $C(p,J)$ is defined in Theorem \ref{Thm_main3}.
\end{corollary}

}

\subsection{Almost sure ensemble collapse}
{We have proven conditions for ensemble collapse in $p$-th moments, but a stronger measure of stochastic convergence is almost sure convergence. This is the focus of this section.}
\begin{theorem}\label{Thm_main4}
{Let $u_0=(u_0^{(j)})_{j\in\{1,\dots,J\}}$ be $\cF_0$-measurable maps $u_0^{(j)}:\Omega\to\cX$} and $\gamma:\R_+\to\R_+$ a positive, monotonically increasing and differentiable function such that $\int_0^{\infty}\frac{\gamma'(s)^2}{\gamma(s)}\,\d s<\infty$.
Then the trivial solution of 
\begin{equation}\label{SDE_mu}
\d \mathfrak e_t^{(j)} = -C(\mathfrak e_t)\mathfrak e_t^{(j)}\d t + C(\mathfrak e_t)\d(W_t^{(j)}-\overline W_t)
\end{equation}
is almost surely asymptotically stable with rate function $\rho(t)=(\gamma(t))^{-\frac12}$. In particular, $(\mathfrak e_t^{(j)})_{j=1,\dots,J}$ converges to zero almost surely as $t\to\infty$.
\end{theorem}

For examples of $\gamma$ see the remark below.
{
\begin{proof}
The idea of this proof is based on Theorem 4.6.2 in \cite{Mao2008}. We define the stochastic Lyapunov function
\begin{equation*}
V(\mathfrak e, t)=\gamma(t)\frac1J\sum\limits_{j=1}^J|\mathfrak e^{(j)}|^2.
\end{equation*}
The generator applied to $V$ fulfills
\begin{eqnarray*}
LV(\mathfrak e, t) 	&= \frac{\gamma'(t)}J\sum\limits_{j=1}^J|\mathfrak e^{(j)}|^2-\gamma(t)\frac{J+1}{J^3}\sum\limits_{j, k=1}^J\langle\mathfrak e^{(k)},\mathfrak e^{(j)}\rangle^2\\
		&\le\frac{\gamma'(t)}J\sum\limits_{j=1}^J|\mathfrak e^{(j)}|^2-\gamma(t)\frac{J+1}{J^3}\sum\limits_{j=1}^J|\mathfrak e^{(j)}|^4.
\end{eqnarray*}
We can maximize this w.r.t. $(|\mathfrak e^{(1)}|^2,\dots,|\mathfrak e^{(J)}|^2)$ and get the following bound for $LV$.
\begin{equation*}
LV(\mathfrak e, t)\le \frac{\gamma'(t)^2}{\gamma(t)}\frac14\frac{J^2}{J+1}=:\eta(t)
\end{equation*}
Since  $\int_0^\infty\eta(t)\,\d t<\infty$, with Theorem 4.6.2 of \cite{Mao2008} the trivial solution of (\ref{SDE_mu}) is almost surely asymptotically stable with rate function $\rho(t)=(\gamma(t))^{-\frac12}$.
\end{proof}

\begin{corollary}
\label{cor2}
Under the same assumptions as in Theorem \ref{Thm_main4} and assumption \eref{A1} it holds true that $(e_t^{(j)})_{j=1,\dots,J}$ converges to zero almost surely as $t\to\infty$ with rate function $\rho(t)=(\gamma(t))^{-\frac12}$.
\end{corollary}

\begin{remark}
Let us give two examples of admissible $\gamma(t)$:
\begin{itemize}
\item
 $\gamma(t)=(t+\varepsilon)^\alpha$ for $\alpha\in(0,1)$ and $\varepsilon>0$ sufficiently small to obtain the rate function $\rho(t) = \frac1{(t+\varepsilon)^{\frac{\alpha}2}}$.
\item
 $\gamma(t)=(t+\varepsilon) \log(t+\varepsilon)^{-\alpha}$ for arbitrarily small $\alpha >\frac12$ and $\varepsilon>0$ to obtain the rate function $\rho(t)=\frac{\log(t+\varepsilon)^{\frac\alpha2}}{(t+\varepsilon)^{\frac12}}$
\end{itemize}
\end{remark}

}
\subsection{Ensemble Collapse in the parameter space}

The following result holds true without the strong assumption \eref{A1}. 
It only shows a monotone decrease, but not the collapse, where we need  \eref{A1}. See also Corollary \ref{cor1} and \ref{cor2}.
\begin{proposition}\label{Prop_Parameter}
Let $u_0=(u_0^{(j)})_{j\in\{1,\dots,J\}}$ be $\cF_0$-measurable maps $u_0^{(j)}:\Omega\to\cX$ such that 
$\E[\frac1J\sum\limits_{j=1}^J|e_0^{(j)}|^2]<\infty$. Then it holds true that
$t\mapsto \E[\frac1J\sum\limits_{j=1}^J|e_t^{(j)}|^2]^{\frac12}$ is monotonically decreasing for $t\geq0$.
\end{proposition}

{
\begin{proof}
It\^{o}'s formula leads to
\begin{eqnarray*}
\d|e_t^{(j)}|^2	&=2\langle e_t^{(j)},\d e_t^{(j)}\rangle+\langle \d e_t^{(j)},\d e_t^{(j)}\rangle\\
			&=-\frac2J\sum\limits_{k=1}^J\langle e_t^{(j)},e_t^{(k)}\rangle\langle \Gamma^{-\frac12}Ae_t^{(k)},\Gamma^{-\frac12}Ae_t^{(j)}\rangle\,\d t\\
			&\ \ \ \ +\frac2J\sum\limits_{k=1}^J\langle e_t^{(j)},e_t^{(k)}\rangle\langle \Gamma^{-\frac12}Ae_t^{(k)},\d(W_t^{(j)}-\overline W_t)\rangle\\
			&\ \ \ \ +\frac{1}{J^2}\sum\limits_{k,l=1}^J\frac{J-1}{J}\langle e_t^{(k)},e_t^{(l)}\rangle\langle \Gamma^{-\frac12}Ae_t^{(k)},\Gamma^{-\frac12}Ae_t^{(l)}\rangle\,\d t
\end{eqnarray*}
and taking the mean over all particles $j\in\{1,\dots,J\}$ gives
\begin{eqnarray*}
\d(\frac1J\sum\limits_{j=1}^J|e_t^{(j)}|^2)&=-\frac{J+1}{J^3}\sum\limits_{j,k=1}^J\langle e_t^{(k)},e_t^{(j)}\rangle\langle \Gamma^{-\frac12}Ae_t^{(k)},\Gamma^{-\frac12}Ae_t^{(j)}\rangle\,\d t\\
			&\ \ \ \ +\frac2{J^2}\sum\limits_{k,j=1}^J\langle e_t^{(k)},e_t^{(j)}\rangle\langle \Gamma^{-\frac12}Ae_t^{(k)},\d(W_t^{(j)}-\overline W_t)\rangle.
\end{eqnarray*}
Again, we do not know, whether the stochastic integral is a martingale, and we need again a localization. 
Consider as in  Lemma \ref{Monotonie} a sequence of stopping times $(\tau_n)_{n\in\N}$ with $\tau_n\to\infty$ a.s., such that 
\begin{equation*}
\int_0^{t\wedge\tau_n}\frac2{J^2}\sum\limits_{k,j=1}^J\langle e_s^{(k)},e_s^{(j)}\rangle\langle \Gamma^{-\frac12}Ae_s^{(k)},\d(W_s^{(j)}-\overline W_s)\rangle
\end{equation*}
is a martingale. We obtain for all $n\in\N$
\begin{eqnarray*}
\fl\E[\frac1J\sum\limits_{j=1}^J|e_{t\wedge\tau_n}^{(j)}|^2] = &\E[\frac1J\sum\limits_{j=1}^J|e_0^{(j)}|^2] 
		-\E[\int_0^{t\wedge\tau_n}\frac{J+1}{J^3}\sum\limits_{j,k=1}^J\langle e_s^{(k)},e_s^{(j)}\rangle\langle \Gamma^{-\frac12}Ae_s^{(k)},\Gamma^{-\frac12}Ae_s^{(j)}\rangle\,\d s]
\end{eqnarray*}
and hence, as we have the positivity of the integrand by  Lemma \ref{nonneg}, we obtain that 
$\E[\frac1J\sum\limits_{j=1}^J|e_{t\wedge\tau_n}^{(j)}|^2]$ is monotonically decreasing and bounded. 
{Analogously to the proof of Lemma \ref{Monotonie}, 
we can pass to the limit $n\to\infty$ by Fatou's lemma and the monotone convergence theorem. 
This implies for $t>s\ge0$
\begin{eqnarray*}
\fl\E[\frac1J\sum\limits_{j=1}^J|e_{t+s}^{(j)}|^2] \le &\E[\frac1J\sum\limits_{j=1}^J|e_s^{(j)}|^2] 
-\E[\int_s^{s+t}\frac{J+1}{J^3}\sum\limits_{j,k=1}^J\langle e_r^{(k)},e_r^{(j)}\rangle\langle \Gamma^{-\frac12}Ae_r^{(k)},\Gamma^{-\frac12}Ae_r^{(j)}\rangle\,\d r]
\end{eqnarray*}
In particular, it follows that $\E[\frac1J\sum\limits_{j=1}^J|e_{t}^{(j)}|^2]$ is monotonically decreasing.}

\end{proof}
}

\section{Convergence to ground truth}
\label{sec5}

{
Under the assumption that $y$ is the image of a truth $u^{\dagger}\in\cX$ under $A$, we are interested now in the analysis of the convergence to the truth.
Recall the equation
\begin{equation*}
\d\mathfrak r_t^{(j)} = -C(\mathfrak r_t)\mathfrak r_t^{(j)}\,\d t+C(\mathfrak r_t)\,\d W_t^{(j)}\,.
\end{equation*}
{The following properties can be shown for the residuals.}
\begin{proposition}
Let $y$ be the image of a truth $u^\dagger\in\cX$ under A and $u_0=(u_0^{(j)})_{j\in\{1,\dots,J\}}$ be $\cF_0$-measurable maps $u_0^{(j)}:\Omega\to\cX$ such that $\E[\frac1J\sum\limits_{j=1}^J|\mathfrak r_0^{(j)}|^2]<\infty$. Then
$\E[\frac1J\sum\limits_{j=1}^J|\mathfrak r_t^{(j)}|^2]^{\frac12}$ is monotonically decreasing.
\end{proposition}
}
{
\begin{proof}
The assertions follow by arguments similar to the proof of Proposition \ref{Prop_Parameter}.
\end{proof}
}

{The main issue in showing convergence of the residuals $\mathfrak r^{(j)}$ to zero is, 
seemingly paradoxically, the fact of ensemble collapse. 
Now obviously, convergence cannot happen without ensemble collapse 
(as the particles cannot converge on the same point if their distance to their joint mean does not vanish) but ensemble collapse itself actually delays convergence. 
To see this, consider the following toy model. It is deterministic, but the same effects can be observed by straightforward extension to an SDE setting.
\begin{eqnarray*}
w'(t) &= -w(t)^3\\
z'(t) &= -w(t)^2 \cdot z(t).
\end{eqnarray*}
Now this system of ODEs can be solved explicitly by separation of variables and it can be seen that for any $w(0) \neq 0$, both $w$ and $z$ will converge to $0$ for $t\to\infty$. It makes sense to try and apply Lyapunov theory with a straightforward Lyapunov functional $V(w,z) = \frac{1}{2}(w^2+z^2)$. Then 
\begin{eqnarray*}
\fl\dot V(w,z) = \langle w(t) \cdot (-w(t)^3) + z(t)\cdot (-w(t)^2\cdot z(t)) = -w(t)^2 \cdot V(w(t),z(t))
\end{eqnarray*}
But $\dot V(w,z)$ is not negatively definite in any neighborhood of $(0,0)$ (the problem being the manifold ${w=0}$). This means we cannot prove that $0$ is an asymptotically stable equilibrium. It actually is not asymptotically stable: If $w(t)$ happens to become $0$ at any time $t$, the whole dynamics will stop there and will not approach $(0,0)$ any further. The origin is rather ``asymptotically stable if bounded away from ${w=0}$'' in a double-cone-like manner.

In a similar way, the solution of the ODE $y'(t) = -t^{-\alpha}y(t)$ will only converge to $0$ if $\alpha \leq 1$, i.e.\ if the rate function does not converge to $0$ too fast. 

The EnKF dynamics works like this toy model: If the ensemble collapse (played by $w$ in the first model and the rate function $t^{-\alpha}$ in the second model) happens too fast, we cannot expect convergence. We suspect that it is possible to prove that the ensemble collapse can be bounded from below (in contrast to also being bounded from above by virtue of theorem \ref{Thm_main3}) and this is the subject of ongoing work. However, the numerical experiments suggest that the collapse happens too fast. In order to circumvent this issue of ``too quick ensemble collapse'' we use artificial inflation of the covariance operator by addition of a positively definite operator (but this is gradually reduced with a certain rate). In addition to solving the problem of counterproductive ensemble collapse, variance inflation stabilizes the convergence in a very suitable manner and is used in practice for this reason, 
see e.g.\ \cite{Evensen2003, 2015arXiv150708319T}.
}
\subsection{Variance Inflation}

In order to correct rank deficiencies of the empirical covariance operator $C(\mathfrak r)$, 
we will use variance inflation in the following sense.
Let $B\in\cL(\R^K,\R^K)$ be a positive definite operator (for example the identity)  
and consider the equation
\begin{equation}\label{lambda_VInf}
\fl\d\mathfrak r_t^{(j)} = -\left(C(\mathfrak r_t)+\frac1{t^{\alpha}+R}B\right)\mathfrak r_t^{(j)}\,\d t+C(\mathfrak r_t)\,\d W_t^{(j)},\quad \alpha\in(0,1),R>0.
\end{equation}
This modification gives convergence of the mapped residuals.
\db{For sufficiently small $\mathfrak r_t$, the new term will dominate, 
and for $\alpha\in(0,1)$ we then expect convergence to $0$ at a rate faster than any polynomial.
The question is now whether and when this asymptotic for small $\mathfrak r_t$ sets in.}
\begin{theorem}\label{Theorem_VInf}
Assume that $y$ is the image of a truth $u^{\dagger}\in\cX$ under $A$ 
and let $\mathfrak r_0 = (\mathfrak r_0^{(j)})_{j\in\{1,\dots,J\}}$ 
be $\cF_0$-measurable maps $\mathfrak r_0^{(j)}:\Omega\to\R^{K}$ 
such that $\E[\frac1J\sum\limits_{j=1}^J|\mathfrak r_0^{(j)}|^2]<\infty$, 
$B\in\cL(\R^K,\R^K)$ a positive definite operator 
and $(\mathfrak r_t^{(j)})_{t\ge0,j=1,\dots,J}$ the solution of (\ref{lambda_VInf}). 
Then \db{for all $\beta>0$} it holds true that
$\E[\frac1J\sum\limits_{j=1}^J|\mathfrak r_t^{(j)}|^2]\in\cO(\db{t^{-\beta}})$ 
and $\E[\frac1J\sum\limits_{j=1}^J|\mathfrak r_t^{(j)}|^2]$ is monotonically decreasing.
\end{theorem}

\begin{proof}
Let $B\in\cL(\R^K,\R^K)$ be a positive definite operator, $\alpha\in(0,1),\ R>0$ and assume, that that the smallest eigenvalue of B is $\lambda_{\mbox{min}} = c>0$.

We derive an equation for $\frac1J\sum\limits_{j=1}^J|\mathfrak r_t^{(j)}|^2$ by using It\^o's formula:
\begin{eqnarray*}
\d|\mathfrak r_t^{(j)}|^2 	&= -2\left\langle\mathfrak r_t^{(j)},\left(C(\mathfrak r_t)+\frac{1}{t^\alpha+R}B\right)\mathfrak r_t^{(j)}\right\rangle\,\d t + 2\langle\mathfrak r_t^{(j)},C(\mathfrak r_t)\d W_t^{(j)}\rangle\\
						&\quad+\frac1J\sum\limits_{j=1}^J\left\langle\mathfrak r_t^{(k)}-\overline{\mathfrak r_t},C(\mathfrak r_t)(\mathfrak r_t^{(k)}-\overline{\mathfrak r_t})\right\rangle\,\d t.
\end{eqnarray*}
Taking the empirical mean over all particles {yields}
\begin{eqnarray*}
\fl\d\frac1J\sum\limits_{j=1}^J|\mathfrak r_t^{(j)}|^2 	&= -\frac2J\sum\limits_{j=1}^J\left\langle\mathfrak r_t^{(j)},\left(C(\mathfrak r_t)+\frac{1}{t^{\alpha}+R}B\right)\mathfrak r_t^{(j)}\right\rangle\,\d t + \frac2J\sum\limits_{j=1}^J\langle\mathfrak r_t^{(j)},C(\mathfrak r_t)\d W^{(j)}\rangle\\
				&\quad +\frac1J\sum\limits_{k=1}^J\langle\mathfrak r_t^{(k)}-\overline{\mathfrak r_t},C(\mathfrak r_t)(\mathfrak r_t^{(k)}-\overline{\mathfrak r}_t)\rangle\,\d t\,.
\end{eqnarray*}
{{Thus, for all $t,s\ge0$, it follows similarly to the proof of Lemma \ref{Monotonie} that}
\begin{eqnarray*}
\fl\E\left[\frac1J\sum\limits_{j=1}^J|\mathfrak r_{t+s}^{(j)}|^2\right]	&\le\E\left[\frac1J\sum\limits_{j=1}^J|\mathfrak r_s^{(j)}|^2\right]-\frac{2}J\int_s^{s+t} \E\left[\sum\limits_{j=1}^J\langle\mathfrak r_r^{(j)},C(\mathfrak r_r)\mathfrak r_r^{(j)}\rangle\right]\,\d r\\
&\quad-\frac2J\int_s^{s+t} \frac{1}{r^{\alpha}+R}\E\left[\sum\limits_{j=1}^J\langle \mathfrak r_r^{(j)},B\mathfrak r_r^{(j)}\rangle\right]\,\d r\\
&\quad+\frac1J\int_s^{s+t} \E\left[\sum\limits_{j=1}^J\langle\mathfrak r_r^{(j)}-\overline{\mathfrak r}_r,C(\mathfrak r_r)(\mathfrak r_r^{(j)}-\overline{\mathfrak r}_r)\rangle\right]\,\d r \\
&\le \E\left[\frac1J\sum\limits_{j=1}^J|\mathfrak r_s^{(j)}|^2\right] - \frac{1}J\int_s^{s+t} \E\left[\sum\limits_{j=1}^J\langle\mathfrak r_r^{(j)},\left(C(\mathfrak r_r)+\frac{1}{r^{\alpha}+R}B\right)\mathfrak r_r^{(j)}\rangle\right]\,\d r\,,
\end{eqnarray*}
where we used Lemma \ref{nonneg} and the non-negativity of $B$.
This yields the monotonicity, as both the covariance $C(\mathfrak{r}_r)$ as well as $B$ are non-negative matrices. }

{
Now we will improve the estimate to obtain the asymptotic rate. 
Consider $S(t)=\frac1J\sum\limits_{j=1}^J|\mathfrak r_{t}^{(j)}|^2$, then 
\[
\d(t^\beta S(t)) =  \beta t^{\beta-1} S(t) dt + t^\beta \d S(t) \;.
\]
Now we can use all the previous estimates for the terms in $\d S$ together with the non-negativity of the covariance matrix $C(\mathfrak r_{t})$
and $B\geq\lambda_{\mbox{min}} >0$ to obtain
\begin{eqnarray*}
 t^\beta \E  S(t) 
 & \leq & \beta \int_0^t \tau^{\beta-1}\E S(\tau) d\tau - \frac2J   \int_0^t \tau^\beta \frac{\lambda_{\mbox{min}}}{\tau^\alpha+R} \E  S(\tau)d\tau \\
 & \leq & \int_0^t \tau^{\beta-1} \Big[ \beta  - \frac{2\lambda_{\mbox{min}}}{J}   \frac{ \tau}{\tau^\alpha+R} \Big] \E  S(\tau)d\tau\;.
\end{eqnarray*}
There is a time $T>0$ such that the integrand in the equation above is negative for all  $t>T$ 
and thus using the monotonicity of  $\E  S(\tau)$ we obtain for all  $t>T$ 
\[
t^\beta \E  S(t) \leq \int_0^T \tau^{\beta-1} \Big[ \beta  - \frac{2\lambda_{\mbox{min}}}{J}   \frac{ \tau}{\tau^\alpha+R} \Big] d\tau \E S(0),
\]
which yields the asymptotic rate $t^{-\beta}$ for $\E  S(t)$.
}
\end{proof}

\begin{remark}
In case of a positive semidefinite matrix $B$, the convergence of the residuals will then take place in the image space of the matrix $B$. 
The proof can be straightforwardly generalized to this setting by projections of the quantities to the corresponding subspace.
\end{remark}

\db{We can also verify almost sure convergence faster than any polynomial rate.}

\begin{theorem}\label{Theorem_VInf2}
Assume that $y$ is the image of a truth $u^{\dagger}\in\cX$ under $A$ 
and let $\mathfrak r_0 = (\mathfrak r_0^{(j)})_{j\in\{1,\dots,J\}}$ 
be $\cF_0$-measurable maps $\mathfrak r_0^{(j)}:\Omega\to\R^{K}$ and $B\in\cL(\R^K,\R^K)$ a positive definite operator. 
Then the solution of \eref{lambda_VInf} is almost surely asymptotically stable with rate function $\rho(t)=t^{-\frac\beta2}$ for all \db{$\beta>0$}. In particular, $(\mathfrak r_t^{(j)})_{j=1,\dots,J}$ converges to zero almost surely as $t\to\infty$.
\end{theorem}

\begin{proof}
We define the Lyapunov function
\begin{equation*}
V(\mathfrak r,t) = t^\beta \frac1J\sum\limits_{j=1}^J|\mathfrak r^{(j)}|^2
\end{equation*}
and obtain
\begin{equation*}
LV(\mathfrak r,t) 
\le \frac{\beta t^{\beta-1}}J\sum\limits_{j=1}^J|\mathfrak r^{(j)}|^2 
- \db{t^\beta} \frac1J\sum\limits_{j=1}^J\langle\mathfrak r^{(j)},\left(C(\mathfrak r)+\frac{1}{t^\alpha+R}B\right)\mathfrak r^{(j)}\rangle.\end{equation*}

Thus,
\begin{equation*}
LV(\mathfrak r, t)
\le \db{\frac1J \sum\limits_{j=1}^J|\mathfrak r^{(j)}|^2\left(\beta -\frac{\lambda_{\min} t}{t^\alpha+R}\right)t^{\beta-1}.}
\end{equation*}

\db{There} is a $T>0$ such that the bracket above is non-positive for all $t\geq T$.
We obtain  $\int_0^\infty LV(\mathfrak r,t)\,\d t\le \int_0^TLV(\mathfrak r,t)\,\d t.$ 
Moreover, by neglecting the negative term in the bracket for $t\leq T$ we obtain
\begin{equation*}
\E[\int_0^T LV(\db{\mathfrak r_t},t)\,\d t] 
\le \E[\int_0^T\db{\beta s^{\beta-1} }\frac1J\sum\limits_{j=1}^J|\db{\mathfrak r^{(j)}_s}|^2\,\d s]
\le \db{\frac{T^\beta}{J}}\E[\sum\limits_{j=1}^J|\mathfrak r_0^{(j)}|^2]<\infty,
\end{equation*}
\db{by using the monotonicity of the sum.}
\db{Hence,} $\int_0^\infty LV(\db{\mathfrak r_t},t)\,\d t\db{<\infty}$  \db{and thus}
$\db{\mathfrak r_t}$ is almost surely asymptotically stable with rate function $\rho(t)=t^{-\frac\beta2}$.
\end{proof}

\cls{
\begin{remark}
Note that the convergence rate is faster than any polynomial rate. However, the proof reveals that the constant in the convergence result will grow w.r.t. the rate $\beta$ and $\alpha\in(0,1)$, which is consistent with the numerical experiments presented in section \ref{sec6}.  
\end{remark}
}

Our aim is to use variance inflation in the parameter space, such that we can apply Theorem \ref{Theorem_VInf}. We will use variance inflation in the finite dimensional system of SDEs of the coordinates in the parameter space.

Let $y\in A\cS$ where $A\cS$ is the linear span of $\{Au_0^{(1)},\dots,Au^{(J)}\}$ and consider the equation
\begin{equation}\label{u_VInf}
\fl\d u_t^{(j)} = (C(u_t)+\frac1{t^\alpha+R}B)A^*\Gamma^{-1}(y-Au_t^{(j)})\,\d t+C(u_t)A^*\Gamma^{-\frac12}\,\d W_t^{(j)},
\end{equation}
$j=1,\dots,J$, for $B$ positive definite, $R>0$ and $\alpha\in(0,1)$. Since $y\in A\cS$, the subspace property still holds, i.e.\ $u_t^{(j)}\in\cS$ for all $(t,j)\in[0,\infty)\times\{1,\dots,J\}$. The following result transfers the results of Theorem \ref{Theorem_VInf} to the parameter space:
\begin{corollary}
Let $y\in A\cS$ and assume that $y$ is the image of a truth $u^\dagger\in\cX$ under $A$, $A^*$ is assumed to be one-to-one and let $(u_t^{(j)})_{t\ge0,j=1,\dots,J}$ be the solution of (\ref{u_VInf}). Then 
\begin{enumerate}
\item $\lim\limits_{t\to\infty}\E[\frac1J\sum\limits_{j=1}^J|\mathfrak e_t^{(j)}|^2] = 0.$
\item $\lim\limits_{t\to\infty}\E[\frac1J\sum\limits_{j=1}^J|\mathfrak r_t^{(j)}|^2] = 0.$
\item $(\mathfrak r_t^{(j)})_{t\ge0}$ converges almost surely to zero with rate function $\rho(t)=t^{-\frac\beta2}$ for all $\beta\db{>0}$.
\end{enumerate}
\end{corollary}
\begin{proof}
{Let $R>0$ and $\alpha\in(0,1)$ 
and observe
\begin{equation*}
\d\mathfrak r_t^{(j)}=-(C(\mathfrak r_t)+\frac{1}{t^\alpha+R}\Gamma^{-\frac12}AB(\Gamma^{-\frac12}A)^*)\mathfrak r_t^{(j)}\,\d t+C(\mathfrak r_t)\,\d W_t^{(j)}.
\end{equation*}
Since $\Gamma^{-\frac12}AB(\Gamma^{-\frac12}A)^*$ is positive definite the second and third assertion follow directly from Theorem \ref{Theorem_VInf} and \ref{Theorem_VInf2}. The proof of the first assertion is similar to the proof of Theorem \ref{Thm_main2}.}
\end{proof}

\section{Numerical Results}\label{sec6}

We consider the problem of recovering the unknown data $u^\dagger$ from noise-free observations $$y^{\dagger}= A(u^\dagger),$$
where $p=\mathcal{A}^{-1}(u)$ is the solution of the one dimensional elliptic equation
\begin{equation}\label{FP}
\eqalign{-\frac{\mathrm{d}^2p}{\mathrm{d} x^2}+p=u \quad &\mbox{in} \ D:=(0,\pi), \cr
	p=0 &\mbox{on $\partial$D}.}
\end{equation}
The forward response operator is defined by 
\[A=\cO\circ \cA^{-1} 
\quad \mathrm{with}\quad 
\cA=-\frac{\mathrm{d}^2}{\mathrm{d} x^2}+id 
\quad\mathrm{on}\quad 
\cD(\cA)=H^2\cap H_0^1
\]
and with operator $\cO$ observing the dynamical system at $K=2^4-1$ equispaced observation points $x_k=\frac k{2^4}$, $k=1,\dots,K$. We approximate the forward-problem \eref{FP} numerically on a uniform mesh with meshwidth $h=2^{-8}$ by a finite element method with continuous, piecewise linear ansatz functions. 

We choose the initial ensemble of particles based on the eigenvalue and eigenfunctions $\{\lambda_j,z_j\}_{j\in\N}$ of the covariance operator $C_0$, defined by $C_0=\beta(\cA-\mathrm{id})^{-1}$ for $\beta=10$. 

From the Bayesian perspective we may interpret this as prior distributed by $\mu_0=\cN(0,C_0)$. We set our $j^{th}$ initial particle to $u^{(j)}(0)=\sqrt{\lambda_j}\zeta_j z_j$ with $\zeta_j\sim\cN(0,1)$, i.e.\ we use the Karhunen-Lo\`eve expansion to generate draws from $\mu_0$. 

The EnKF continuous time limit 
\begin{equation*}
\mathrm{d}u_t^{(j)} = C(u_t)A^*\Gamma^{-1}(y-Au_t^{(j)})\,\mathrm{d}t+C(u_t)A^*\Gamma^{-\frac12}\,\mathrm{d}W_t^{(j)},
\end{equation*}
is discretized by equation \eref{Iteration} for the following simulations.

\paragraph{Ensemble collapse}
In the following we illustrate the results from section \ref{Asymp_beh}, in particular the bounds on the ensemble collapse derived in Theorem \ref{Thm_main2} and in Theorem \ref{Thm_main3}.

\begin{figure}[!htb]
    \subfigure{\includegraphics[width=0.5\textwidth]{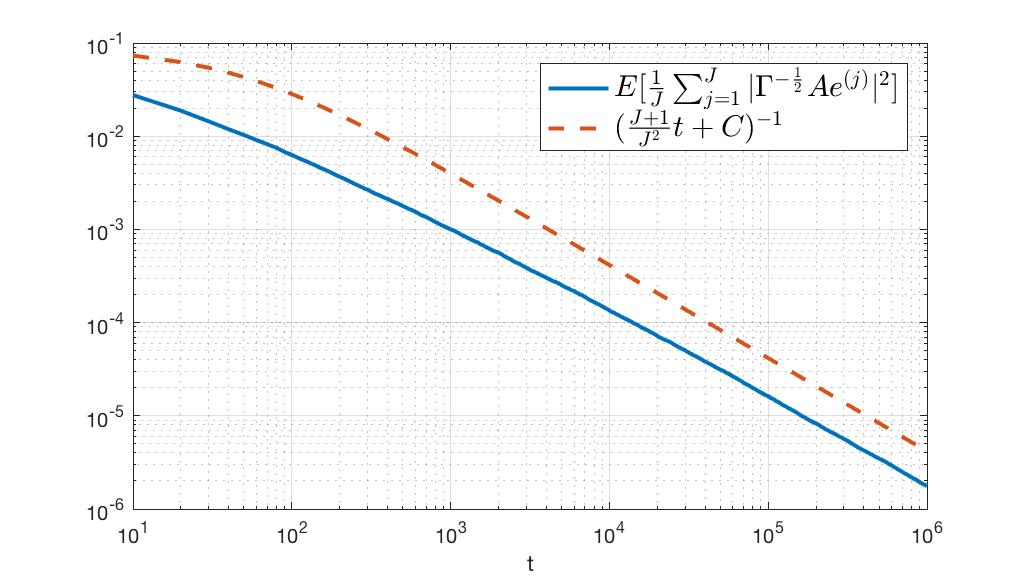}} 
    \subfigure{\includegraphics[width=0.5\textwidth]{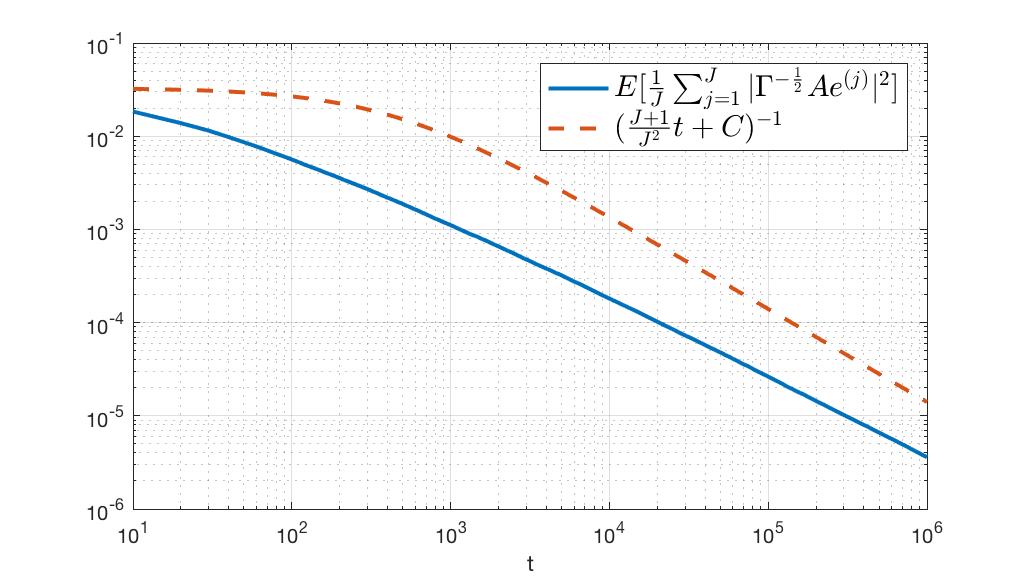}}
    \caption{$\hat{E}(\frac1J\sum\limits_{j=1}^J|\mathfrak e^{(j)}(t)|^2)$ with w.r. of time. $Q=1000$ paths with $J=5$ (left) and $J=15$ (right) particles has been simulated.}\label{Ae}
\end{figure}

Figure \ref{Ae} shows that the Monte Carlo approximation of the expected value $\hat{E}[\frac1J\sum\limits_{j=1}^J|\mathfrak e_t^{(j)}|^2]$ is bounded from above by $(\frac{J+1}{J^2}t+C)^{-1}$ with $C=(\hat{E}[\frac1J\sum\limits_{j=1}^J|\mathfrak e_0^{(j)}|^2])^{-1}$, as derived in Theorem \ref{Thm_main2}.

\begin{figure}[!htb]
    \subfigure{\includegraphics[width=0.5\textwidth]{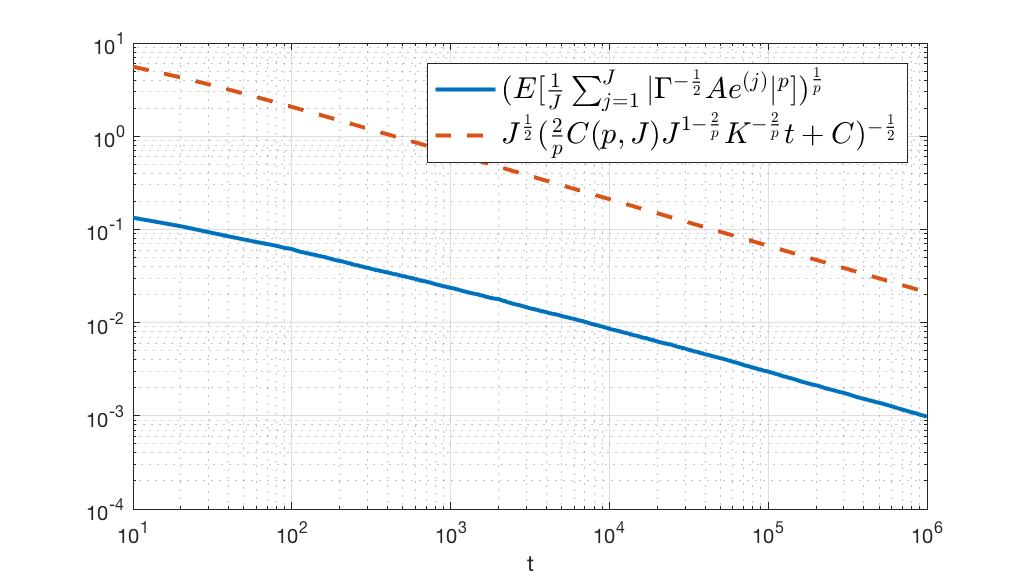}} 
    \subfigure{\includegraphics[width=0.5\textwidth]{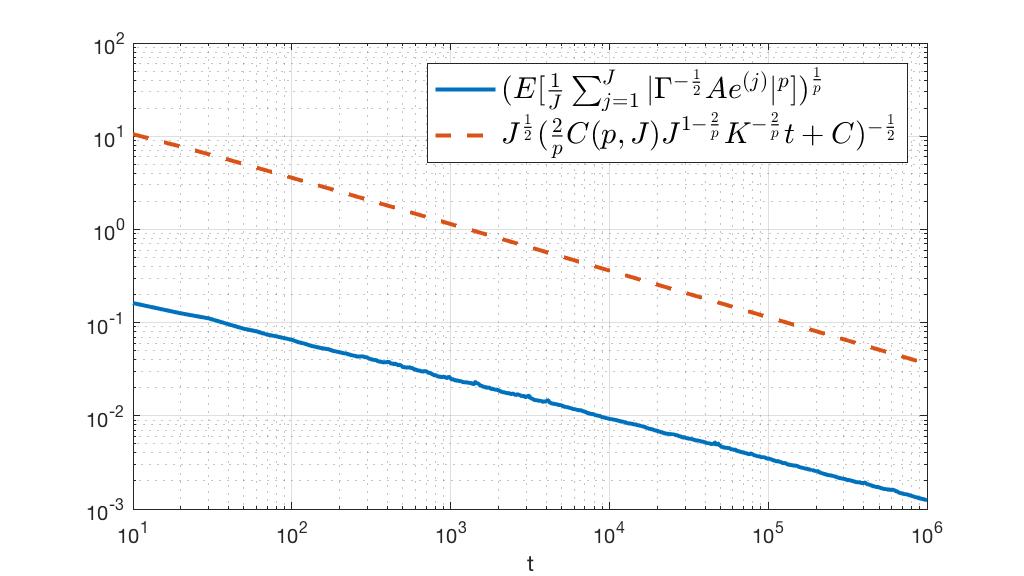}}
    \caption{$\hat{E}(\frac1J\sum\limits_{j=1}^J|\mathfrak e^{(j)}(t)|^{p})^{-\frac1{p}}$, $p=\lfloor \frac{J+3}2\rfloor-1$, w.r. of time. $Q=1000$ paths with $J=5$ (left) and $J=15$ (right) particles has been simulated.}\label{Aep}
\end{figure} 

Similarly Figure \ref{Aep} demonstrates that the approximated higher moments $\hat E[\frac1J\sum\limits_{j=1}^J|\mathfrak e_t^{(j)}|^{p}]^{-\frac1{p}}$ are bounded by $J^{\frac12}(\frac2pC(p,J)J^{1-\frac2p}K^{-\frac2p}t+C)^{-\frac12}$ with $C =(K^{\frac{p-1}2}\hat E[\frac1J\sum\limits_{j=1}^J|\mathfrak e_0^{(j)}|^p])^{\frac2p}$, compare Theorem \ref{Thm_main3}.

In order to verify the almost sure ensemble collapse numerically, we have simulated $Q=10$ paths.
\begin{figure}[!htb]
    \subfigure{\includegraphics[width=0.5\textwidth]{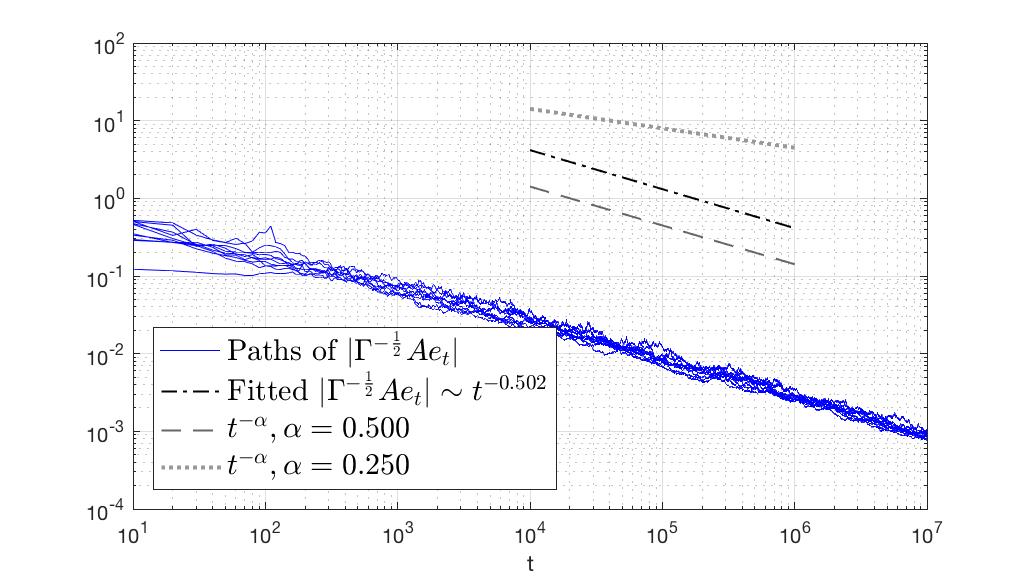}} 
    \subfigure{\includegraphics[width=0.5\textwidth]{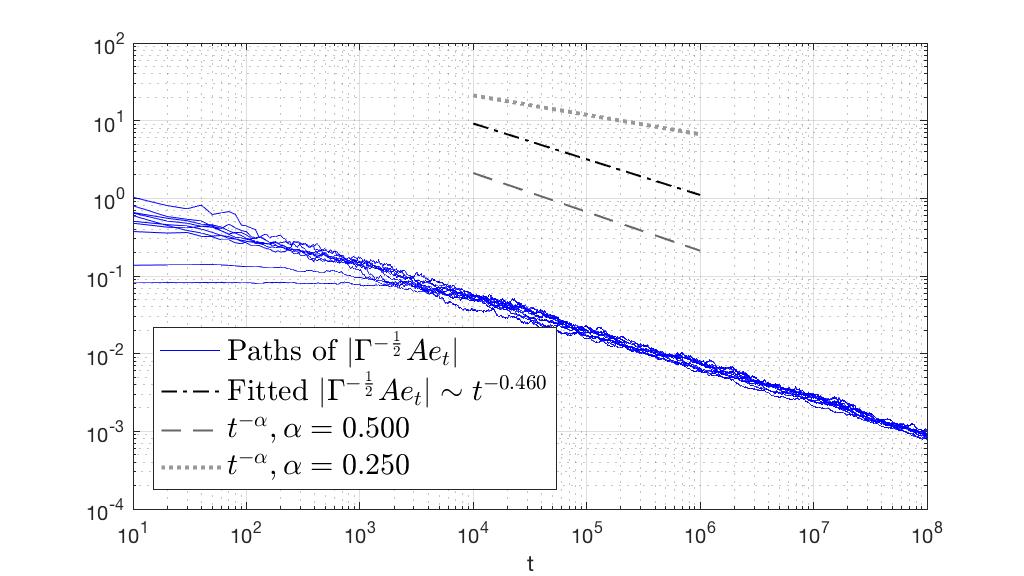}}
    \caption{Paths of $|\mathfrak e(t)|_2$ w.r. of time. Q=10 paths with $J=5$ (left) and $J=15$ (right) particles has been simulated.}\label{paths}
\end{figure} 

From Theorem \ref{Thm_main4} we know, that $\mathfrak e(t)$ converges almost surely to zero with rate function $\rho(t) = t^{-\frac \alpha2}$ for every $\alpha\in(0,1)$. Figure \ref{paths} illustrates this behavior, the expected convergence rates can be observed in this example.

\paragraph{Convergence to ground truth}
We compare simulations of the ensemble Kalman inversion without variance inflation with simulations of the ensemble Kalman inversion with variance inflation. The variance inflation is used in the following setting: We set $\alpha\in\{\frac12, \frac34\}$ and $R=1$ in equation \eref{u_VInf}. 
The number of particles is $J=15$, 
i.e.\ the forward response operator is bijective as a mapping from the subspace spanned by the initial ensemble to the data space.

\begin{figure}[!htb]
    \subfigure{\includegraphics[width=0.49\textwidth]{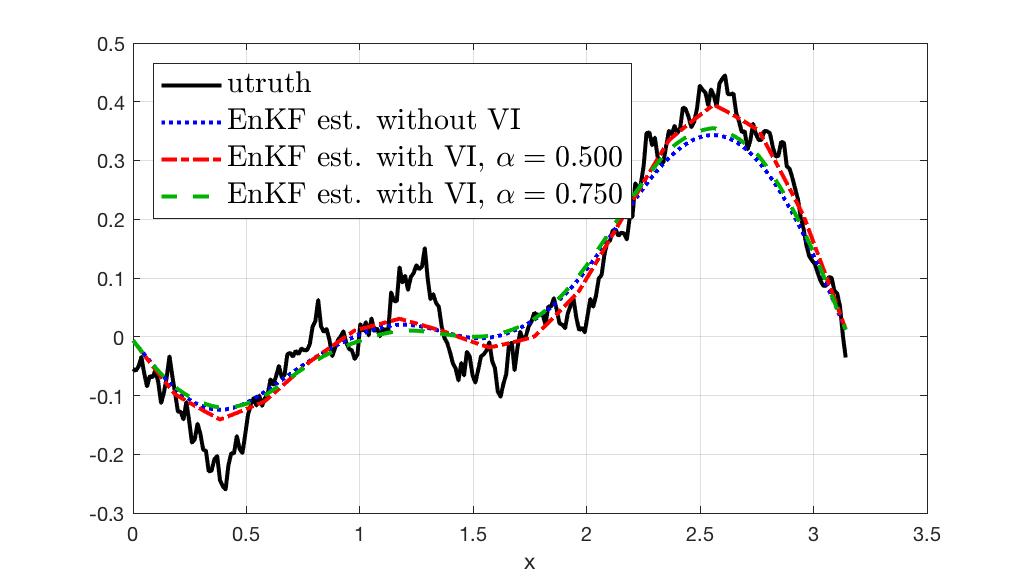}} 
    \subfigure{\includegraphics[width=0.49\textwidth]{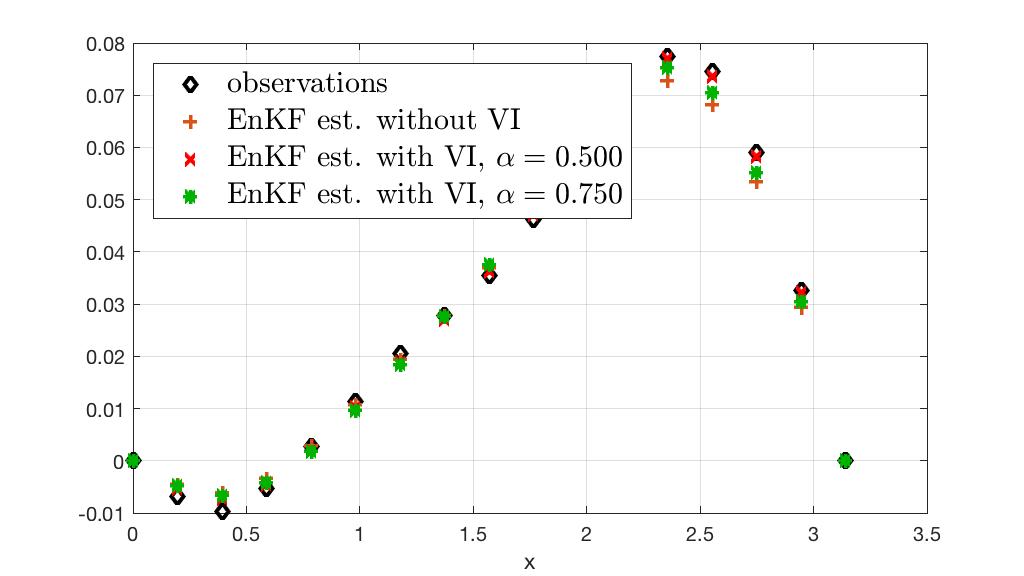}}
    \caption{EnKF estimation without VI vs. EnKF estimation with VI. J=15 particles and Q=1000 paths has been simulated.}\label{EnKF_VI}
\end{figure} 

Figure \ref{EnKF_VI} shows the differences of the EnKF estimation in the parameter space as well as in the observation space. 
We observe that the simulations with variance inflation giving a better estimation in the observation space as well as in the parameter space. 
If we reduce the variance inflation in time faster, i.e.\ we increase the parameter $\alpha$ from $\frac12$ to $\frac34$, the effect of the variance inflation decreases. The following figures demonstrate the effect on the ensemble collapse and the residuals.

\begin{figure}[!htb]
    \subfigure{\includegraphics[width=0.49\textwidth]{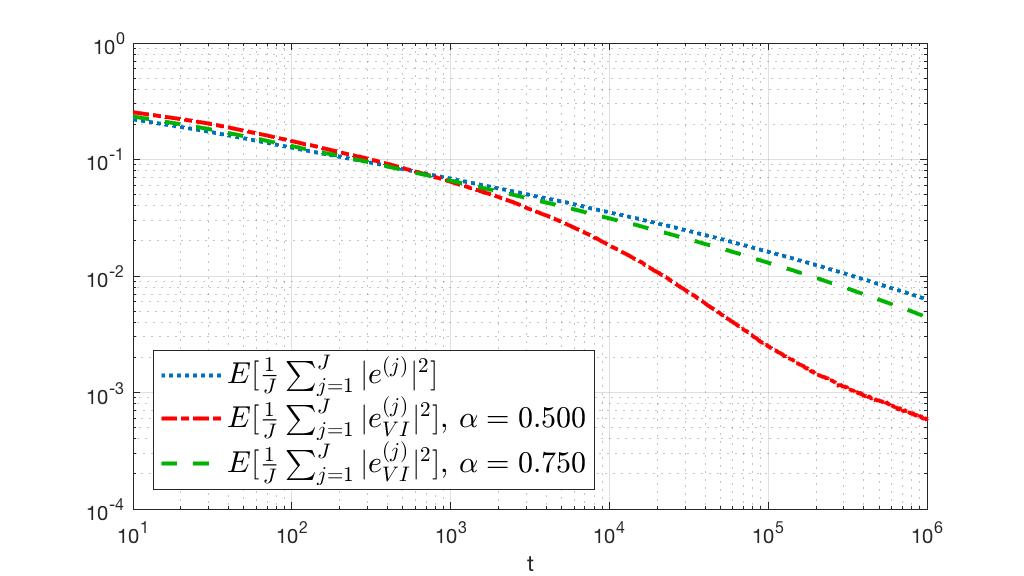}} 
    \subfigure{\includegraphics[width=0.49\textwidth]{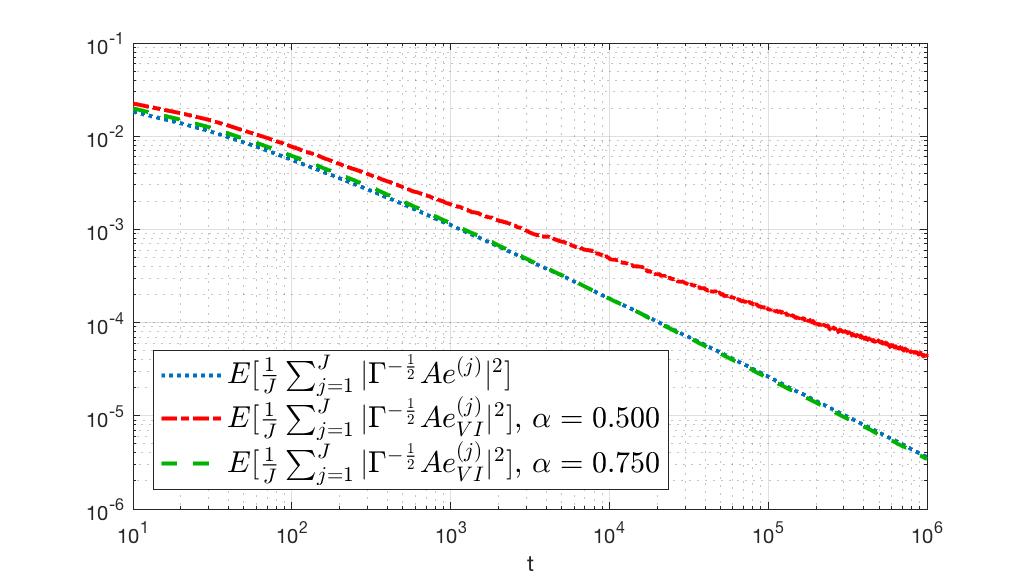}}
    \caption{Comparison of the spread of the ensemble w.r. to time with VI and without VI.}\label{en_col_VI}
\end{figure}

The idea of the variance inflation was to slow down the convergence of the particles to the ensemble mean, i.e.\ to control the rate of the ensemble collapse, in order to ensure the convergence of the residuals in the observation space. Figure \ref{en_col_VI} illustrates that we can ensure a higher spread of the ensemble in the simulations with variance inflation in comparison to the simulations without variance inflation in the observation space.

\begin{figure}[!htb]
    \subfigure{\includegraphics[width=0.49\textwidth]{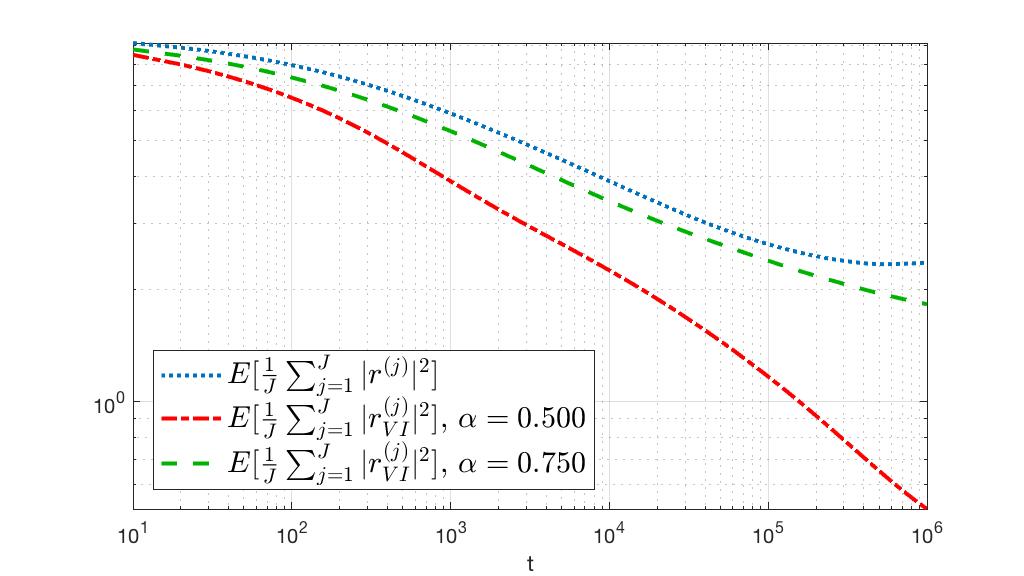}} 
    \subfigure{\includegraphics[width=0.49\textwidth]{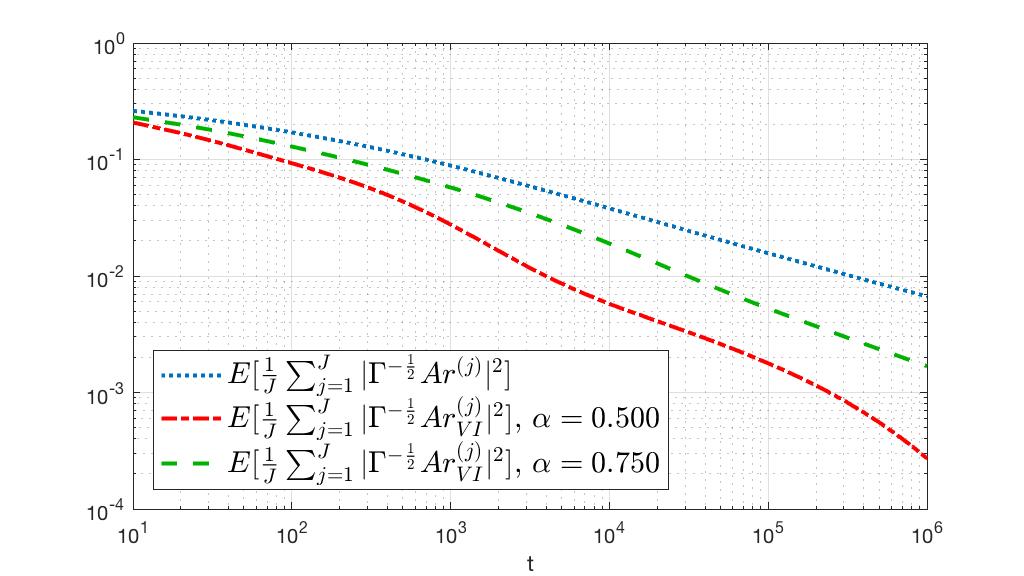}}
    \caption{Comparison of the residuals w.r. to time with VI and without VI.}\label{residuals_VI}
\end{figure} 

Figure \ref{residuals_VI} points out that we end up with convergence of the residuals in the observation and parameter space in case of variance inflation. Without variance inflation the simulations show a slight increase of the residuals in the parameter space, suggesting that the convergence of the residuals will slow down in the observation space as well.

\sw{

To emphasize this result, we reduce the dimension of the example and we set $h=2^4$ with $K=3$ equispaced observation points. Furthermore, we set again $R=1$ and $\alpha=\frac12$ and we use $J=3$ particles, such that the forward response operator is again bijective as mapping from the subspace spanned by the initial ensemble to the observation space.

\begin{figure}[!htb]
    \subfigure{\includegraphics[width=0.49\textwidth]{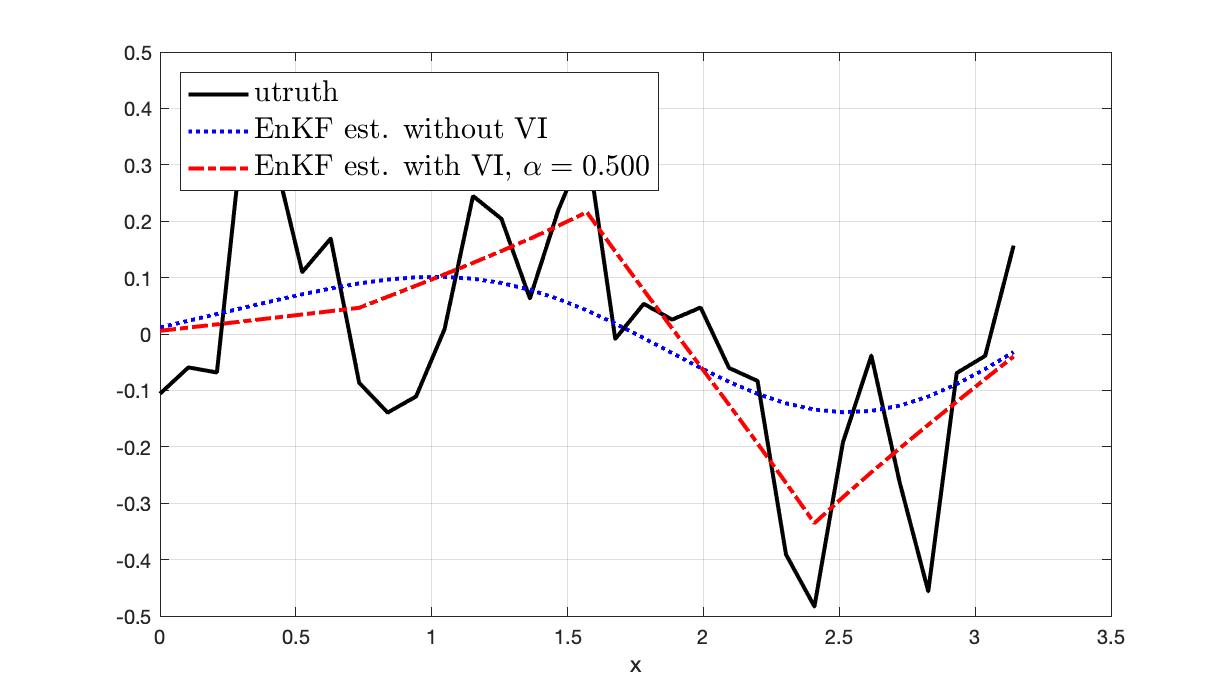}} 
    \subfigure{\includegraphics[width=0.49\textwidth]{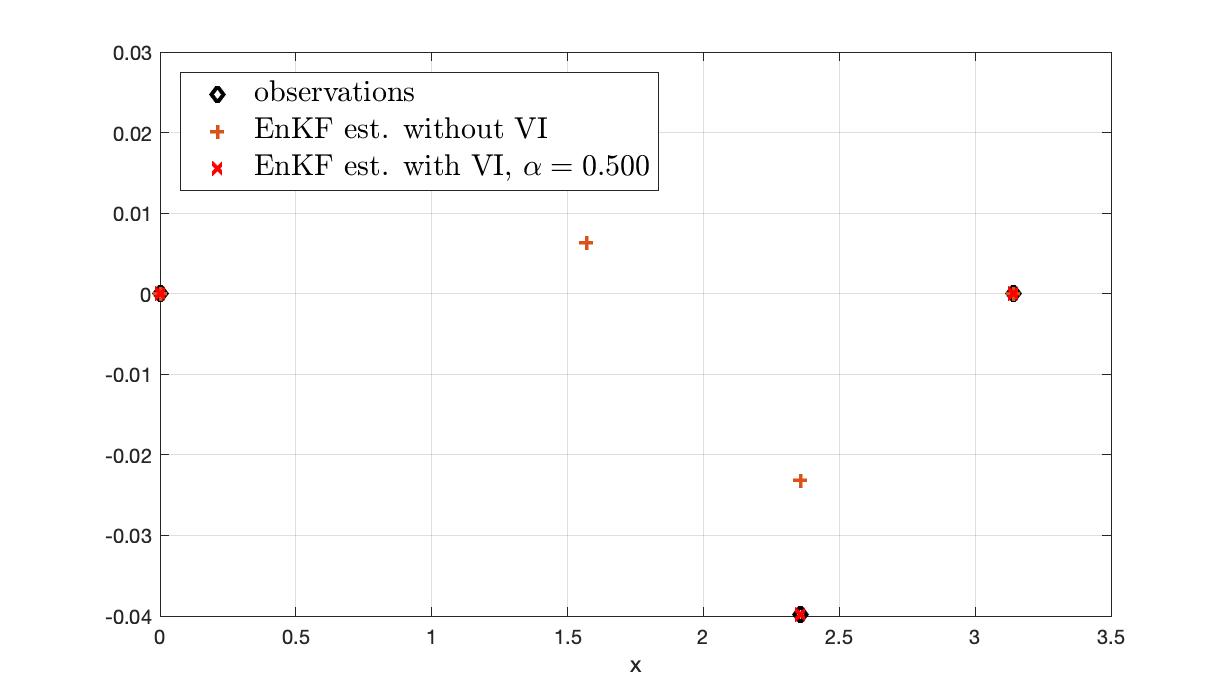}}
    \caption{EnKF estimation without VI vs. EnKF estimation with VI. J=3 particles and Q=10000 paths has been simulated.}\label{EnKF_VI_lowdim}
\end{figure} 

Figure \ref{EnKF_VI_lowdim} shows again the difference of the EnKF estimation with and without variance inflation

\begin{figure}[!htb]
    \subfigure{\includegraphics[width=0.49\textwidth]{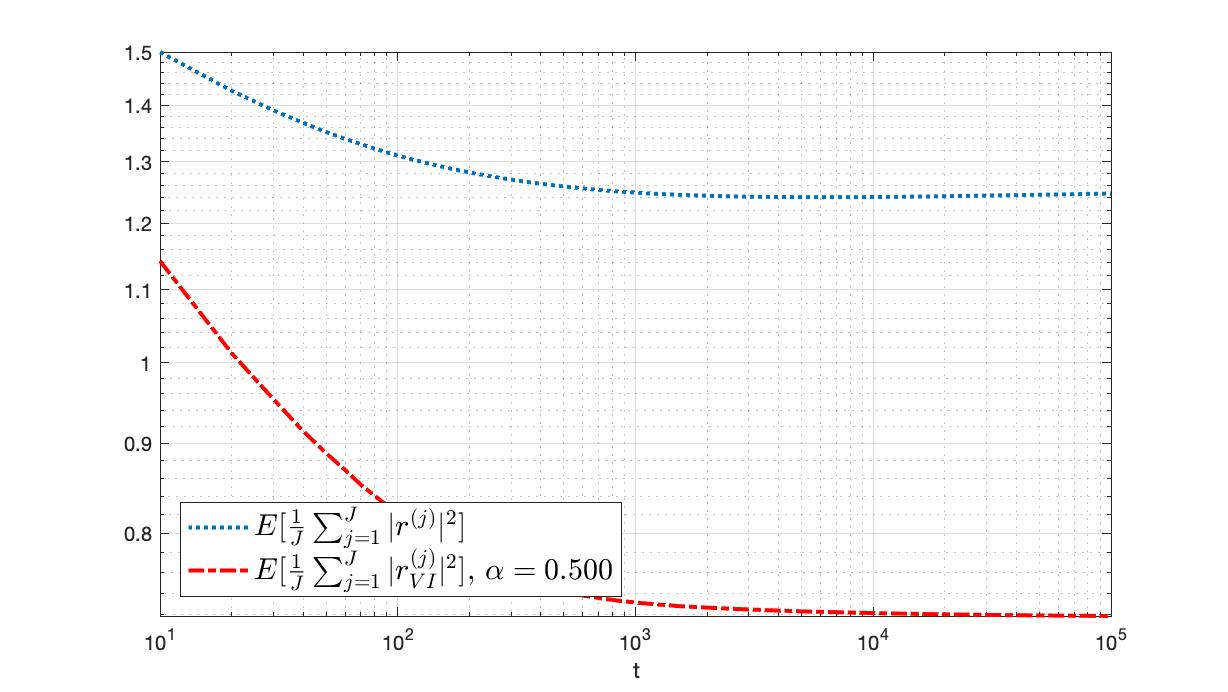}} 
    \subfigure{\includegraphics[width=0.49\textwidth]{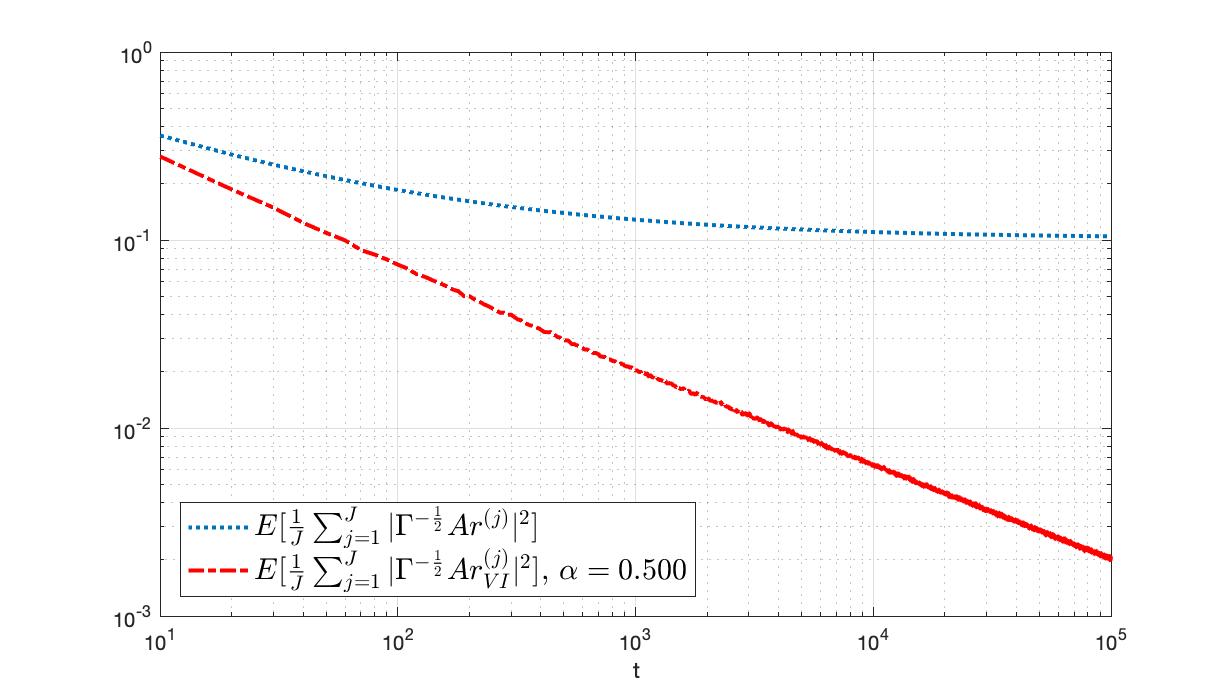}}
    \caption{Comparison of the residuals w.r. to time with VI and without VI.}\label{residuals_VI_lowdim}
\end{figure} 

\begin{figure}[!htb]
    \subfigure{\includegraphics[width=0.49\textwidth]{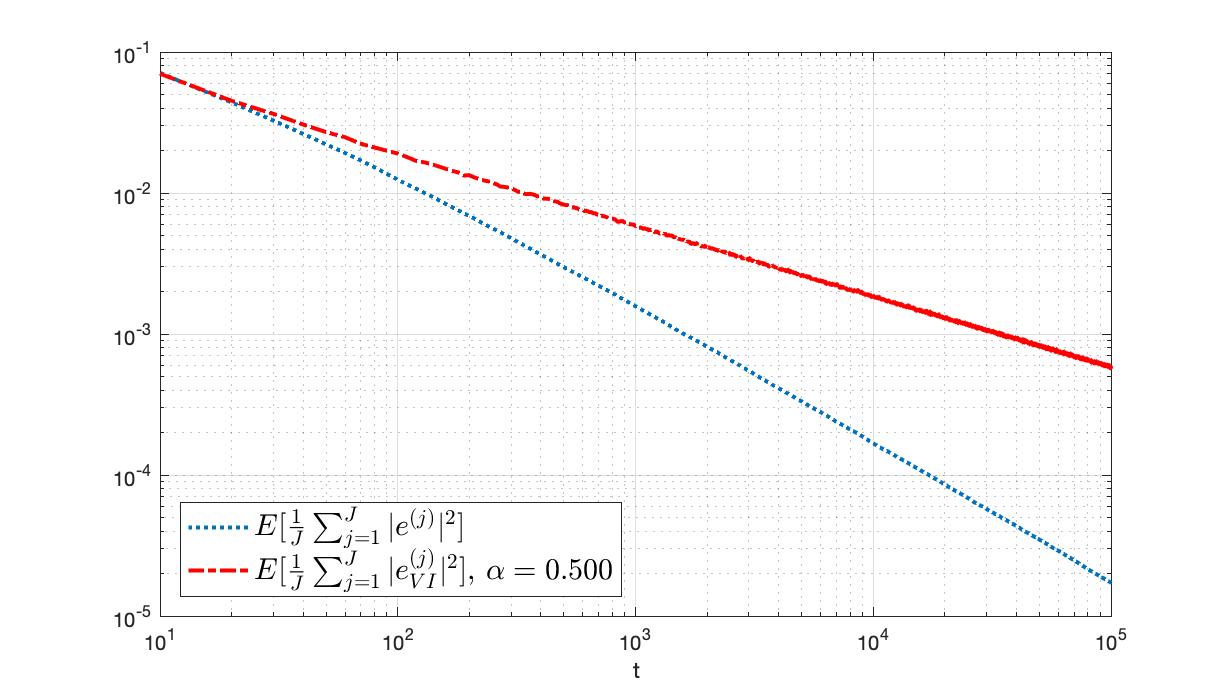}} 
    \subfigure{\includegraphics[width=0.49\textwidth]{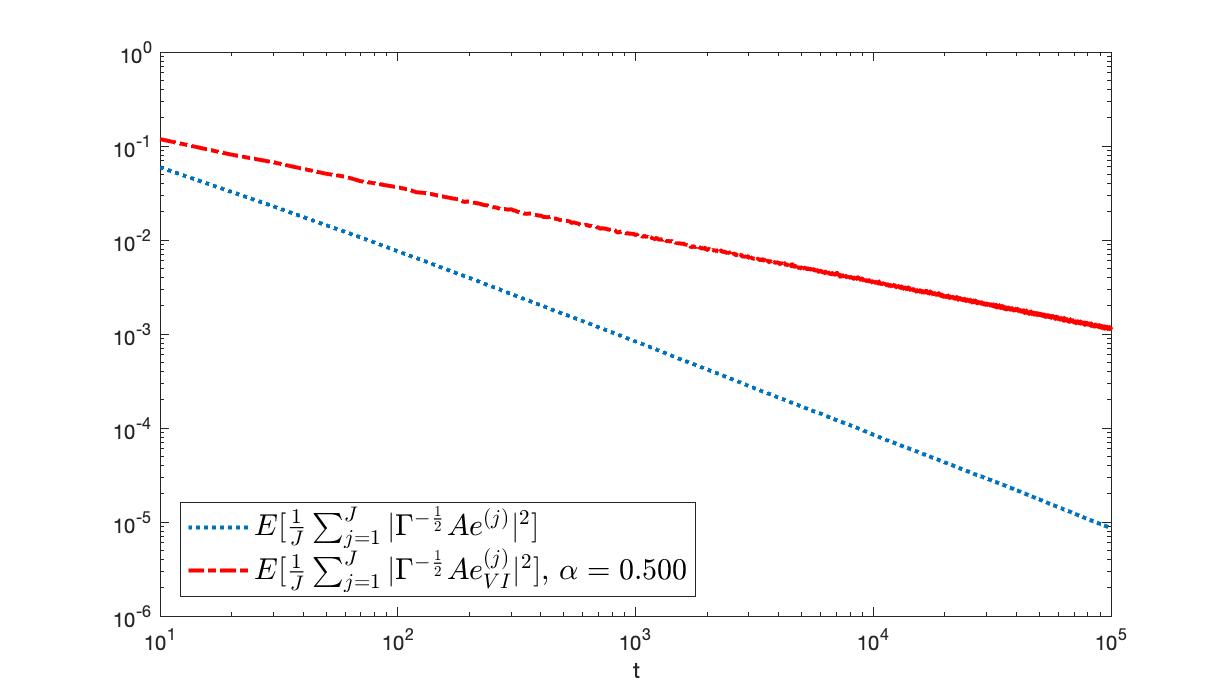}}
    \caption{Comparison of the ensemble spread w.r. to time with VI and without VI.}\label{spread_VI_lowdim}
\end{figure} 

Figure \ref{residuals_VI_lowdim} points out the effect of the variance inflation. While the residuals in the observation space without variance inflation diverge, \cls{we obtain convergence of the residuals in the observation space using variance inflation}. In addition, in Figure \ref{spread_VI_lowdim} we can see that the ensemble of particles still collapse in the \cls{parameter space} as well as in the observation space.

}

\section{Conclusions}\label{sec:concl}
Our analysis of the ensemble Kalman inversion shows the well-posedness and accuracy of the method in the case of linear forward operators. The results are based on the continuous time limit of the algorithm consisting of a coupled system of stochastic differential equations. Due to the subspace property of the ensemble Kalman inversion, the theory of finite-dimensional stochastic differential equations could be applied to establish existence and uniqueness of solutions, 
i.e.\ to show the well-posedness of the method. The ensemble collapse has been quantified in terms of moments as well as almost sure convergence of the particles to the empirical mean. Furthermore, we suggest a time-adaptive variance inflation to stabilize the convergence of the empirical mean to the truth in the noise free case. The inflation can be interpreted as model error delaying the ensemble collapse. The presented numerical experiments confirm the theoretical results and indicate that the ensemble collapse can be bounded from below for the original iteration scheme without variance inflation. However, the rate seems to be too small to achieve convergence. This will be subject to future work. \cls{In addition, the next steps include the generalization of the presented results to case of noisy observations in the inverse problem and the development of appropriate stopping criteria in the noisy case.} Even though the presented analysis relies on the linearity of the forward operator, the statements hold true for non-Gaussian priors and can guide the analysis of the nonlinear setting.

\noindent{\bf Acknowledgments} ClS would like to thank the Isaac Newton Institute for Mathematical Sciences for support and hospitality during the programme Uncertainty quantification for complex systems: theory and methodologies when work on this paper was undertaken. This work was supported by: EPSRC grant numbers EP/K032208/1 and EP/R014604/1". SW is grateful to the DFG RTG1953 "Statistical Modeling of Complex Systems and Processes" for funding of this research. 
The authors acknowledge support by the state of Baden-W\"urttemberg through bwHPC.\\[0.55cm]

\appendix

\section{Auxiliary results}

\label{sec:app}
{In order to use It\^o's formula we have to calculate the following quadratic covariation in many cases:}
\begin{lemma}\label{BM_innerproduct}
Let $(W^{(j)})_{j=1,\dots,J}$ be independent Brownian motions in $\R^K${, $u,v\in\R^K$} and let $l\neq j\in\{1,\dots,J\}$. {Then with $\bar W = \frac{1}{J}\sum_{k=1}^JW^{(k)}$,}
\begin{equation*}
\langle u,d(W^{(j)}-\overline{W})\rangle\langle v,d(W^{(j)}-\overline{W})\rangle = \frac{J-1}{J}\langle u,v\rangle\,dt,
\end{equation*}
\begin{equation*}
\langle u,d(W^{(j)}-\overline{W})\rangle\langle v,d(W^{(l)}-\overline{W})\rangle = -\frac{1}{J}\langle u,v\rangle\,dt.
\end{equation*}
\end{lemma}
\begin{proof}
Observe
\begin{equation*}
W^{(j)}-\overline W = -\frac1J\sum\limits_{k=1, k\neq j}^JW^{(k)}+\frac{J-1}JW^{(j)}
\end{equation*}
Since $W^{(k)}$ are independent Brownian motions it follows
\begin{eqnarray*}
\langle u,d(W^{(j)}-\overline{W})\rangle\langle v,d(W^{(j)}-\overline{W})\rangle &= \frac1{J^2}\sum\limits_{k=1, k\neq j}^J\langle u,\d W^{(k)}\rangle\langle v,\d W^{(k)}\rangle\\
				&\quad+\frac{(J-1)^2}{J^2}\langle u,\d W^{(j)}\rangle\langle v,\d W^{(j)}\rangle\\
				&=\frac{J-1}J\langle u,v\rangle\,\d t
\end{eqnarray*}
Similarly,
\begin{eqnarray*}
\fl \langle u,d(W^{(j)}-\overline{W})\rangle\langle v,d(W^{(l)}-\overline{W})\rangle &=-\frac1J\sum\limits_{k=1}^J(\langle u,\d W^{(j)}\rangle\langle v,\d W^{(k)}\rangle + \langle u,\d W^{(k)}\rangle\langle v,\d W^{(l)}\rangle)\\
		&\quad+\frac1{J^2}\sum\limits_{i, k=1}^J\langle u,\d W^{(i)}\rangle\langle v,\d W^{(k)}\rangle\\
		&=-\frac1J\langle u,v\rangle\,\d t
\end{eqnarray*}
\end{proof}

\begin{lemma}\label{nonneg}
Let $M$ be a symmetric and nonnegative $d\times d$-matrix, then for all choices of vectors $(z^{(k)})_{k=1,\dots,J}$ in $\R^n$ we have 
\begin{equation*}
\sum\limits_{k, l=1}^J\langle z^{(k)},z^{(l)}\rangle\langle z^{(k)},Mz^{(l)}\rangle\ge0.
\end{equation*}
\end{lemma}
\begin{proof}
Let $(v^{(m)})_{m=1,\dots,d}$ be an orthonormal basis of eigenvectors such that $Mv^{(m)}=\lambda_mv^{(m)}$ with $\lambda_m\ge0$. Then $z^{(l)}=\sum\limits_{m=1}^dz_m^{(l)}v^{(m)}$ and thus
\begin{equation*}
\fl\sum\limits_{k, l=1}^J\langle z^{(k)},z^{(l)}\rangle\langle z^{(k)},Mz^{(l)}\rangle=\sum\limits_{k, l=1}^J\sum\limits_{m, n=1}^dz_n{(k)}z_n^{(l)}z_m^{(k)}z_m^{(l)}\lambda_m = \sum\limits_{n, m=1}^d\lambda_m(\sum\limits_{k=1}^Jz_n^{(k)}z_m^{(k)})^2\ge0.
\end{equation*}
\end{proof}

\begin{lemma}\label{Lemma_App}
Let $(x^{(j)})_{j=1,\dots,J}$ be vectors in $\R^n$ and let $C(x)$ denote the sample covariance matrix
\begin{equation*}
C(x)=\frac1J\sum\limits_{k=1}^J(x^{(k)}-\overline{x})\otimes(x^{(k)}-\overline{x}),\qquad \overline x=\frac1J\sum\limits_{j=1}^Jx^{(j)}.
\end{equation*}
Then {it holds true that}
\begin{equation*}
\sum\limits_{j=1}^J\langle x^{(j)}-\overline{x},C(x)(x^{(j)}-\overline{x})\rangle\le\sum\limits_{j=1}^J\langle x^{(j)},C(x)x^{(j)}\rangle
\end{equation*}
\end{lemma}
\begin{proof}
By expanding the non-centered quadratic form we obtain
\begin{eqnarray*}
\fl\sum\limits_{j=1}^J\langle x^{(j)}-\overline{x},C(x)(x^{(j)}-\overline{x})\rangle
&=&\sum\limits_{j=1}^J\langle x^{(j)},C(x)x^{(j)}\rangle - J\langle\overline{x},C(x)\overline{x}\rangle
\,,
\end{eqnarray*}
which yields the claim by the non-negativity of the covariance matrix.
\end{proof}

\begin{lemma}\label{martingale}
For all $j\in\{1,\dots,J\}$ the process 
\begin{equation*}
(M(t))_{t\ge0}:=\Big(\int_0^t\mathfrak e_s^{(j)T}C(\mathfrak e_s)\,\d  W_s^{(j)}\Big)_{t\ge0}
\end{equation*} 
is a (global) martingale.
\end{lemma}
\begin{proof}
The local martingale given by the stochastic integral is a true martingale by It\^o-isometry
if we show that following second moment is finite 
(cp.\cite[Theorem 2.4]{GL2011})
\begin{equation*}
\fl\|\mathfrak e_\cdot^{(j)T}C(\mathfrak e_{\cdot})\|_{\Lambda_2;T}
:=\E[\int_0^T\|\mathfrak e_s^{(j)T}C(\mathfrak e_s)\|_F^2\,\d s]
=\int_0^T\E[\|\mathfrak e_s^{(j)T}C(\mathfrak e_s)\|_F^2]\,\d s
<\infty
\end{equation*} 
for all $T\ge0$. For this, we first estimate the Frobenius norm by
\begin{eqnarray*}
\fl\|\mathfrak e_s^{(j)T}C(\mathfrak e_s)\|_F^2 := \trace{\mathfrak e^{(j)T}C(\mathfrak e)(\mathfrak e^{(j)T}C(\mathfrak e))^T} &= \frac1{J^2}\sum\limits_{k,l=1}^J\langle \mathfrak e^{(l)},\mathfrak e^{(k)}\rangle\langle \mathfrak e^{(j)},\mathfrak e^{(k)}\rangle\langle \mathfrak e^{(l)},\mathfrak e^{(j)}\rangle\\
&\le\frac1{J^2}\sum\limits_{k,l=1}^J|\mathfrak e^{(l)}|^2|\mathfrak e^{(j)}|^2|\mathfrak e^{(k)}|^2
\end{eqnarray*}
Thus, it holds true that \begin{equation*}\fl\frac1J\sum\limits_{j=1}^J\|\mathfrak e_s^{(j)T}C(\mathfrak e_s)\|_F^2\le\frac1{J^3}\sum\limits_{j,k,l=1}^J|\mathfrak e^{(l)}|^2|\mathfrak e^{(j)}|^2|\mathfrak e^{(k)}|^2 = (\frac1J\sum\limits_{j=1}^J|e^{(j)}|^2)^3\le \frac1J\sum\limits_{j=1}^J|e^{(j)}|^6\end{equation*}
and 
with Lemma \ref{Monotonie} it follows
\begin{equation*}
\frac1J\sum\limits_{j=1}^J\|\mathfrak e_\cdot^{(j)T}C(\mathfrak e_{\cdot})\|_{\Lambda_2;T}
\le\int_0^T\E[\frac1J\sum\limits_{j=1}^J|\mathfrak e^{(j)}|^6]\,\d s
\le C,
\end{equation*}
since $p+2 := 6 \le J+4$.
\end{proof}

\begin{lemma}\label{martingale2}
For all $k\in\{1,\dots,J\}$ and $p\in(2,\frac{J+3}2)$ the process 
\begin{equation*}
(M(t))_{t\ge0} :=\left(\int_0^t\frac{p}{J^2}\sum\limits_{m=1}^K((\sum\limits_{k=1}^J|\mathfrak e_m^{(k)}|^2)^{\frac p2-1}\sum\limits_{j, l=1}^J\mathfrak e_m^{(l)}\mathfrak e_m^{(j)})\mathfrak e^{(l)\top}\d W^{(k)}\right)
\end{equation*}
is a (global) martingale.
\end{lemma}
\begin{proof}
Similarly to the proof of Lemma \ref{martingale} we estimate the Frobenius norm of the integrand by
\begin{eqnarray*}
\fl\| \sum\limits_{m=1}^K((\sum\limits_{k=1}^J|\mathfrak e_m^{(k)}|^2)^{\frac p2-1}\sum\limits_{j, l=1}^J\mathfrak e_m^{(l)}\mathfrak e_m^{(j)})\mathfrak e^{(l)\top}\|_F^2&\le C_1(J)\sum\limits_{m=1}^K(\sum\limits_{k=1}^J|\mathfrak e_m^{(k)}|^2)^{p-2}\sum\limits_{j, l=1}^J(\mathfrak e_m^{(l)})^2(\mathfrak e_m^{(j)})^2|\mathfrak e^{(l)}|^2\\
			&\le C_2(J,K) \sum\limits_{k=1}^J|\mathfrak e^{(k)}|^{2(p-2)}\sum\limits_{j, l=1}^J |\mathfrak e^{(l)}|^4|\mathfrak e^{(j)}|^2\\
			&\le C_3(J,K)\sum\limits_{l=1}^J|\mathfrak e^{(l)}|^{2p+2},
\end{eqnarray*}
where we have used Jensen's inequality and the fact $|\mathfrak e_m^{(j)}|^2 \le \sum\limits_{n=1}^K|\mathfrak e_n^{(j)}|^2 = |\mathfrak e^{(j)}|^2$. 
The assertion follows by the bound \eref{bound_p+2} in the proof of Theorem \ref{Thm_main3}, which we obtained by localization and Fatou's Lemma without martingale property.
\end{proof}
{
\section{Higher-order ensemble collapse: Proof of Theorem \ref{Thm_main3}}\label{sec:app2}
We will use the following auxiliary result in order to prove Theorem \ref{Thm_main3}.
It is a well known statement of the equivalence of norms, but we need the precise constants.
\begin{lemma}\label{lem:pTransformation}
For $a_{m,j}\in \R$, $m=1,\ldots, d$, $j=1,\ldots,J$ and $p\in \N$,
\begin{displaymath}
\sum_{j=1}^J(\sum_{m=1}^d|a_{m,j}|^2)^\frac{p}{2} 
\leq d^{(p-1)/2} \cdot \sum_{m=1}^d\sum_{j=1}^J |a_{m,j}|^p
\end{displaymath}
and
\[\sum_{m=1}^d\sum_{j=1}^J |a_{m,j}|^p \leq  J^{p/2}\cdot \sum_{m=1}^d(\sum_{j=1}^J|a_{m,j}|^2)^\frac{p}{2}.
\]
By symmetry we also have
\begin{displaymath}
\fl\sum_{m=1}^d(\sum_{j=1}^J|a_{m,j}|^2)^\frac{p}{2} \leq J^{p/2} \cdot \sum_{m=1}^d\sum_{j=1}^J |a_{m,j}|^p \quad \mbox{ and }\quad \sum_{m=1}^d\sum_{j=1}^J |a_{m,j}|^p  \leq d^{(p-1)/2} \cdot \sum_{j=1}^J(\sum_{m=1}^d|a_{m,j}|^2)^\frac{p}{2}.
\end{displaymath}
\end{lemma}
\begin{proof}
We start with the first claim and write
\[\sum_{j=1}^J(\sum_{m=1}^d|a_{m,j}|^2)^\frac{p}{2} = \sum_{j=1}^J T_j\]
with $T_j^2 = (\sum_{m=1}^d |a_{m,j}|^2)^p$. We continue by expressing $T_j^2$ using the multinomial theorem and Young's inequality
\begin{eqnarray*}
T_j^2 &= \sum_{k_1+\cdots+k_d = p}\Big(\begin{array}{cc}
p\\k_1,\ldots, k_d
\end{array}\Big) \cdot \prod_{m=1}^d |a_{m,j}|^{2\cdot k_m}\\
&= \sum_{k_1+\cdots+k_d = p}\Big(\begin{array}{cc}
p\\k_1,\ldots, k_d
\end{array}\Big)\cdot \prod_{m=1,k_m\neq 0}^d |a_{m,j}|^{2\cdot k_m}\\
&\le  \sum_{m=1}^d |a_{m,j}|^{2p} \cdot \sum_{l_1+\cdots + l_d = p-1}\Big(\begin{array}{cc}
p-1\\l_1,\ldots, l_d
\end{array}\big) = \sum_{m=1}^d |a_{m,j}|^{2p} \cdot d^{p-1}.
\end{eqnarray*}
This means that
\[\sum_{j=1}^J(\sum_{m=1}^d|a_{m,j}|^2)^\frac{p}{2} \leq d^\frac{p-1}{2}\cdot \sum_{j=1}^J \sqrt{\sum_{m=1}^d |a_{m,j}|^{2p}} \leq d^\frac{p-1}{2}\cdot \sum_{j=1}^J  \sum_{m=1}^d|a_{m,j}|^p,\]
which proves the first statement. For the second claim we can write by concavity of the square root
\begin{equation*}
\sum_{m=1}^d(\sum_{j=1}^J|a_{m,j}|^2)^\frac{p}{2} = \sum_{m=1}^d(\sqrt{J}\cdot\sqrt{\sum_{j=1}^J\frac{|a_{m,j}|^2}{J}})^p
\geq J^{-\frac{p}{2}}\sum_{m=1}^d\sum_{j=1}^J|a_{m,j}|^p,
\end{equation*}
i.e. 
\begin{displaymath}
\sum_{m=1}^d\sum_{j=1}^J|a_{m,j}|^p \leq J^{\frac{p}{2}}\cdot \sum_{m=1}^d(\sum_{j=1}^J|a_{m,j}|^2)^\frac{p}{2}
\end{displaymath}
\end{proof}

\begin{proof}[Proof of Theorem \ref{Thm_main3}]
Recall the equation of $\mathfrak e^{(j)}$
\begin{equation*}
\d \mathfrak e^{(j)} = - \frac{1}{J}\sum_{l=1}^J \mathfrak e^{(l)} \langle \mathfrak e^{(l)}, \mathfrak e^{(j)}\rangle \d t + \frac{1}{J}\sum_{l=1}^J \mathfrak e^{(l)} \langle \mathfrak e^{(l)},\d (W^{(j)}-\overline W)\rangle.
\end{equation*}
And (recall that $\mathfrak e^{(j)} \in\R^K$) componentwise
\begin{equation*}
 \d \mathfrak e_m^{(j)} = -\frac{1}{J}\sum_{l=1}^J\mathfrak e_m^{(l)}\langle\mathfrak e^{(l)},\mathfrak e^{(j)}\rangle\d t + \frac{1}{J}\sum_{l=1}^J\mathfrak e_m^{(l)}\langle\mathfrak e^{(l)},\d(W^{(j)}-\overline W)\rangle.
 \end{equation*}
 
We define the Lyapunov function (for equivalent notions of ``$p$-norms'' of the ensemble, see lemma \ref{lem:pTransformation})
\begin{equation*}
V_p(\mathfrak e) = \frac{1}{J}\sum_{m=1}^K(\sum_{j=1}^J|\mathfrak e_m^{(j)}|^2)^\frac{p}{2}
\end{equation*}
and according to Ito's lemma it holds that
\[\d V_p(\mathfrak e) = \sum_{m=1}^K\sum_{j=1}^J \frac{\partial V_p}{\partial \mathfrak e_m^{(j)}}\d \mathfrak e_m^{(j)} + \frac{1}{2}\sum_{m,m'=1}^K\sum_{j,j'=1}^J \d \mathfrak e_m^{(j)}  \frac{\partial^2 V_p}{\partial \mathfrak e_m^{(j)} \partial \mathfrak e_{m'}^{(j')}} \d\mathfrak e_{m'}^{(j')}\]

{Analogously to the proof of Theorem \ref{Thm_main2} the expectation is given by 
\begin{equation}\label{bound_p+1}
\fl\eqalign{\E[V_p(\mathfrak e_{s+t})]&=\E[V_p(\mathfrak e_s)]\cr
					&- C(p,J)\E[\int_s^{s+t} \sum_{m=1}^K [\{(\sum_{k=1}^J |\mathfrak e_m^{(k)}|^2)^{\frac{p}{2}-1} \}[\sum_{n=1}^K(\sum_{l=1}^J\mathfrak e_m^{(l)}\mathfrak e_n^{(l)})^2]]\,\d r]\cr
					&+\E[\int_0^t\frac{p}{J^2}\sum_{m=1}^K ((\sum_{k=1}^J |\mathfrak e_m^{(k)}|^2)^{\frac{p}{2}-1} \sum_{j,l=1}^J \mathfrak e_m^{(l)}\mathfrak e_m^{(j)} \langle \mathfrak e^{(l)},\d (W^{(j)} - \frac{1}{J}\sum_{r=1}^J W^{(r)})\rangle)]}
\end{equation}
by defining $C(p,J):=\frac{p}{J^2} (1 - \frac{(p-2+J)\cdot(J-1)}{2J^2}-\frac{p-2}{2J^2})$.

Thus, similarly to Lemma \ref{Monotonie} we obtain by setting $s=0$ and using Fatou's Lemma
\begin{equation*}
\E[V_p(\mathfrak e_0)]\ge C(p,J)\E[\int_0^t \sum_{m=1}^K [\{(\sum_{k=1}^J |\mathfrak e_m^{(k)}|^2)^{\frac{p}{2}-1} \}[\sum_{n=1}^K(\sum_{l=1}^J\mathfrak e_m^{(l)}\mathfrak e_n^{(l)})^2]]\,\d s].
\end{equation*}
}

Note that
\begin{equation*}
\E[\int_0^t \sum_{m=1}^K [\{(\sum_{k=1}^J |\mathfrak e_m^{(k)}|^2)^{\frac{p}{2}-1} \}[\sum_{n=1}^K(\sum_{l=1}^J\mathfrak e_m^{(l)}\mathfrak e_n^{(l)})^2]]\,\d s]<C
\end{equation*}
Now we bound the integrand by below by:
\begin{equation*}
\sum_{m=1}^K ((\sum_{k=1}^J |\mathfrak e_m^{(k)}|^2)^{\frac{p}{2}-1})(\sum_{n=1}^K(\sum_{l=1}^J\mathfrak e_m^{(l)}\mathfrak e_n^{(l)})^2)
\ge \sum\limits_{m=1}^K(\sum\limits_{k=1}^J|\mathfrak e_m^{(k)}|^2)^{\frac p2+1} = J V_{p+2}(\mathfrak e),
\end{equation*}
Thus, we also have
\begin{equation}\label{bound_p+2}
\E[\int_0^t V_{p+2}(\mathfrak e_s)\,\d s]<C
\end{equation}
for all $p<J+3$.

Note, that with (\ref{bound_p+2}) one can prove similar to Lemma \ref{martingale}, that the stochastic integral
\begin{equation*}
\int_0^t\frac{p}{J^2}\sum_{m=1}^K ((\sum_{k=1}^J |\mathfrak e_m^{(k)}|^2)^{\frac{p}{2}-1} \sum_{j,l=1}^J \mathfrak e_m^{(l)}\mathfrak e_m^{(j)} \langle \mathfrak e^{(l)},\d (W^{(j)} - \frac{1}{J}\sum_{r=1}^J W^{(r)})\rangle)
\end{equation*}
is a martingale for all $p\in(2,\frac{J+3}2)$. For details see Lemma \ref{martingale2}.

By \eref{bound_p+1} we get that $\E[V_p(\mathfrak e_t)]$ is monotonically decreasing and it follows 
\begin{eqnarray*}
\E[V_p(\mathfrak e_t)]\le \E[V_p(\mathfrak e_0)]-C(p,J)J\int_0^t\E[V_{p+2}(\mathfrak e_s)]\,ds.
\end{eqnarray*}
By Jensen's inequality it follows
\begin{eqnarray*}
V_{p+2}(\mathfrak e) = \frac1J\sum\limits_{m=1}^K(\sum\limits_{j=1}^J|\mathfrak e_m^{(j)}|^2)^{\frac p2\frac{p+2}p}\ge K^{-\frac2p}J^{-\frac2p}(V_p(\mathfrak e))^{\frac{p+2}p}
\end{eqnarray*}
and we obtain
\begin{equation*}
\E[V_p(\mathfrak e_t)]\le \E[V_p(\mathfrak e_0)]-C(p,J)J^{1-\frac2p}K^{-\frac2p}\int_0^t\E[V_{p}(\mathfrak e_s)]^{\frac{p+2}p}\,ds.
\end{equation*}
Similarly to the proof of Theorem \ref{Thm_main2} we get
\begin{equation*}
h'\le-C(p,J)J^{1-\frac2p}K^{-\frac2p}h^{\frac{p+2}p},
\end{equation*}
by defining $h(t):=\E[V_p(\mathfrak e_t)]$, from which it follows that
\begin{equation*}
h(t)\le (\frac2pC(p,J)K^{-\frac2p}J^{1-\frac2p}t+(h(0))^{-\frac2p})^{-\frac p2}.
\end{equation*}

Finally, we conclude with
\begin{equation*}
\E[\frac1J\sum\limits_{j=1}^J|\mathfrak e_t^{(j)}|^p]\le J^{\frac p2}(\frac2pC(p,J)K^{-\frac2p}J^{1-\frac2p}t+(K^{\frac{p-1}2}\E[\frac1J\sum\limits_{j=1}^J|\mathfrak e_0^{(j)}|^p])^{-\frac2p})^{-\frac p2}
\end{equation*}
by using Lemma \ref{lem:pTransformation}.
\end{proof}

}

\section*{References}
\nocite{*}
\bibliographystyle{unsrt.bst}
\bibliography{mybib}

\begin{thebibliography}{10}

\bibitem{Evensen2003}
Geir Evensen.
\newblock The {E}nsemble {K}alman filter: theoretical formulation and practical
  implementation.
\newblock {\em Ocean Dynamics}, 53(4):343--367, Nov 2003.

\bibitem{oliver2008inverse}
Dean~S. Oliver, Albert~C. Reynolds, and Ning Liu.
\newblock {\em Inverse theory for petroleum reservoir characterization and
  history matching}.
\newblock Cambridge University Press, 2008.

\bibitem{doi:10.1002/2017GL076101}
Tapio Schneider, Shiwei Lan, Andrew Stuart, and Joao Teixeira.
\newblock Earth system modeling 2.0: A blueprint for models that learn from
  observations and targeted high-resolution simulations.
\newblock {\em Geophysical Research Letters}, 44(24):12,396--12,417, 2017.

\bibitem{HU2012145}
Jiatang Hu, Katja Fennel, Jann~Paul Mattern, and John Wilkin.
\newblock Data assimilation with a local ensemble {K}alman filter applied to a
  three-dimensional biological model of the middle atlantic bight.
\newblock {\em Journal of Marine Systems}, 94:145 -- 156, 2012.

\bibitem{4914783}
Mark~D. Butala, Richard~A. Frazin, Yuguo Chen, and Farzad Kamalabadi.
\newblock Tomographic imaging of dynamic objects with the ensemble {Kalman}
  filter.
\newblock {\em IEEE Transactions on Image Processing}, 18(7):1573--1587, July
  2009.

\bibitem{DESIMON2018220}
Lia~De Simon, Marco Iglesias, Benjamin Jones, and Christopher Wood.
\newblock Quantifying uncertainty in thermophysical properties of walls by
  means of bayesian inversion.
\newblock {\em Energy and Buildings}, 177:220 -- 245, 2018.

\bibitem{Iglesias_2018}
Marco Iglesias, Minho Park, and M~V Tretyakov.
\newblock Bayesian inversion in resin transfer molding.
\newblock {\em Inverse Problems}, 34(10):105002, jul 2018.

\bibitem{Nikola}
Nikola Kovachki and Andrew~M. Stuart.
\newblock Ensemble {Kalman} inversion: A derivative-free technique for machine
  learning tasks.
\newblock {\em ArXiv e-prints}, August 2018.

\bibitem{L2009LargeSA}
François Le~Gland, Valerie Monbet, and Vu-Du Tran.
\newblock {Large sample asymptotics for the ensemble {{{Kalman}}} filter}.
\newblock Research Report RR-7014, {INRIA}, 2009.

\bibitem{doi:10.1137/140965363}
Evan Kwiatkowski and Jan Mandel.
\newblock Convergence of the square root ensemble {Kalman} filter in the large
  ensemble limit.
\newblock {\em SIAM/ASA Journal on Uncertainty Quantification}, 3(1):1--17,
  2015.

\bibitem{doi:10.1137/140984415}
Kody Law, Hamidou Tembine, and Raul Tempone.
\newblock Deterministic mean-field ensemble {Kalman} filtering.
\newblock {\em SIAM Journal on Scientific Computing}, 38(3):A1251--A1279, 2016.

\bibitem{doi:10.1137/15M100955X}
Haakon Hoel, Kody Law, and Raul Tempone.
\newblock Multilevel ensemble {Kalman} filtering.
\newblock {\em SIAM Journal on Numerical Analysis}, 54(3):1813--1839, 2016.

\bibitem{2016arXiv160808558C}
Alexey {Chernov}, Haakon {Hoel}, Kody {Law}, Fabio {Nobile}, and Raul
  {Tempone}.
\newblock {Multilevel ensemble {Kalman} filtering for spatially extended
  models}.
\newblock {\em ArXiv e-prints}, August 2016.

\bibitem{0951-7715-27-10-2579}
David Kelly, Kody Law, and Andrew~M. Stuart.
\newblock Well-posedness and accuracy of the ensemble {K}alman filter in
  discrete and continuous time.
\newblock {\em Nonlinearity}, 27(10):2579, 2014.

\bibitem{47c5c78ebb8c44ef9d8d5e0d23e23c13}
{Xin T.} Tong, {Andrew J.} Majda, and David Kelly.
\newblock Nonlinear stability of the ensemble {K}alman filter with adaptive
  covariance inflation.
\newblock {\em Communications in Mathematical Sciences}, 14(5):1283--1313,
  2016.

\bibitem{0951-7715-29-2-657}
David Kelly, Andrew~J. Majda, and Xin~T. Tong.
\newblock Nonlinear stability and ergodicity of ensemble based {K}alman
  filters.
\newblock {\em Nonlinearity}, 29(2):657, 2016.

\bibitem{doi:10.1002/cpa.21722}
Andrew~J. Majda and Xin~T. Tong.
\newblock Performance of ensemble {K}alman filters in large dimensions.
\newblock {\em Communications on Pure and Applied Mathematics}, 71(5):892--937,
  2018.

\bibitem{Tong2018}
Xin~T. Tong.
\newblock Performance analysis of local ensemble {Kalman} filter.
\newblock {\em Journal of Nonlinear Science}, 28(4):1397--1442, Aug 2018.

\bibitem{delmoral2018}
Pierre Del~Moral and Julian Tugaut.
\newblock On the stability and the uniform propagation of chaos properties of
  ensemble {K}alman {B}ucy filters.
\newblock {\em The Annals of Applied Probability}, 28(2):790--850, 04 2018.

\bibitem{doi:10.1137/17M1119056}
Jana de~Wiljes, Sebastian Reich, and Wilhem Stannat.
\newblock Long-time stability and accuracy of the ensemble {Kalman}--bucy
  filter for fully observed processes and small measurement noise.
\newblock {\em SIAM Journal on Applied Dynamical Systems}, 17(2):1152--1181,
  2018.

\bibitem{ErnstEtAl2015}
Oliver~G. Ernst, Bj{\"o}rn Sprungk, and Hans-J{\"o}rg Starkloff.
\newblock Analysis of the ensemble and polynomial chaos {K}alman filters in
  {B}ayesian inverse problems.
\newblock {\em SIAM/ASA Journal on Uncertainty Quantification}, 3(1):823--851,
  2015.

\bibitem{SchSt2016}
Claudia Schillings and Andrew~M. Stuart.
\newblock Analysis of the ensemble {K}alman filter for inverse problems.
\newblock {\em SIAM Journal on Numerical Analysis}, 55(3):1264--1290, 2017.

\bibitem{bergemann2010localization}
Kay Bergemann and Sebastian Reich.
\newblock A localization technique for ensemble {{{Kalman}}} filters.
\newblock {\em Quarterly Journal of the Royal Meteorological Society},
  136(648):701--707, 2010.

\bibitem{bergemann2010mollified}
Kay Bergemann and Sebastian Reich.
\newblock A mollified ensemble {{{Kalman}}} filter.
\newblock {\em Quarterly Journal of the Royal Meteorological Society},
  136(651):1636--1643, 2010.

\bibitem{Reich2011}
Sebastian Reich.
\newblock A dynamical systems framework for intermittent data assimilation.
\newblock {\em BIT Numerical Mathematics}, 51(1):235--249, Mar 2011.

\bibitem{Iglesias2015}
Marco~A. Iglesias.
\newblock Iterative regularization for ensemble data assimilation in reservoir
  models.
\newblock {\em Computational Geosciences}, 19(1):177--212, Feb 2015.

\bibitem{2016InvPr..32b5002I}
Marco~A. Iglesias.
\newblock A regularizing iterative ensemble {K}alman method for
  {PDE}-constrained inverse problems.
\newblock {\em Inverse Problems}, 32(2):025002, 2016.

\bibitem{bloemker2018strongly}
Dirk Bl{\"o}mker, Claudia Schillings, and Philipp Wacker.
\newblock A strongly convergent numerical scheme from ensemble kalman
  inversion.
\newblock {\em SIAM Journal on Numerical Analysis}, 56(4):2537--2562, 2018.

\bibitem{doi:10.1080/00036811.2017.1386784}
Claudia Schillings and Andrew~M. Stuart.
\newblock Convergence analysis of ensemble {Kalman} inversion: the linear,
  noisy case.
\newblock {\em Applicable Analysis}, 97(1):107--123, 2018.

\bibitem{Dean2008}
Y.~Zhang, N.~Liu, and D.S. Oliver.
\newblock Ensemble filter methods with perturbed observations applied to
  nonlinear problems.
\newblock {\em Comput Geosciences}, 14(2), 2010.

\bibitem{StLawIg2013}
Marco~A. Iglesias, Kody Law, and Andrew~M. Stuart.
\newblock Ensemble {K}alman methods for inverse problems.
\newblock {\em Inverse Problems}, 29(4):045001, 2013.

\bibitem{Hasminski}
Rafail~Z. Khasminskii.
\newblock {\em Stochastic stability of differential equations. Transl. by D.
  Louvish. Ed. by S. Swierczkowski.}
\newblock Monographs and Textbooks on Mechanics of Solids and Fluids.
  Mechanics: Analysis, 7. Alphen aan den Rijn, The Netherlands; Rockville,
  Maryland, USA. Sijthoff \& Noordhoff, 1980.

\bibitem{Mao2008}
Xuerong Mao.
\newblock {\em Stochastic Differential Equations and Applications}.
\newblock Horwood series in mathematics \& applications. Horwood Pub., 2008.

\bibitem{2015arXiv150708319T}
David Kelly, Andrew~J. Majda, and Xin~T. Tong.
\newblock {Nonlinear stability of the ensemble {Kalman} filter with adaptive
  covariance inflation}.
\newblock {\em ArXiv e-prints}, July 2015.

\bibitem{GL2011}
Leszek Gawarecki.
\newblock {\em Stochastic Differential Equations in Infinite Dimensions with
  Applications to Stochastic Partial Differential Equations}.
\newblock Probability and Its Applications. Springer Berlin Heidelberg, Berlin,
  Heidelberg, 2011.

\bibitem{Liu15}
Wei Liu and Michael R\"ockner.
\newblock {\em Stochastic Partial Differential Equations: An Introduction}.
\newblock Universitext. Springer, Cham, 1st ed. 2015 edition, 2015.

\bibitem{Kha2012}
Rafail~Z. Chasʹminskij.
\newblock {\em Stochastic stability of differential equations}.
\newblock Stochastic Modelling and Applied Probability; 66. Springer,
  Heidelberg [u.a.], compl. rev. and enl. 2. ed. edition, 2012.

\bibitem{LawKody2015Da:a}
Kody Law, Andrew~M. Stuart, and Konstantinos Zygalakis.
\newblock {\em Data Assimilation: A Mathematical Introduction}.
\newblock Texts in Applied Mathematics. Springer International Publishing,
  2016.

\bibitem{Kelly201511063}
David Kelly, Andrew~J. Majda, and Xin~T. Tong.
\newblock Concrete ensemble {K}alman filters with rigorous catastrophic filter
  divergence.
\newblock {\em Proceedings of the National Academy of Sciences}, 2015.

\bibitem{2014arXiv1411.4608H}
El~{houcine Bergou}, Serge {Gratton}, and Jan {Mandel}.
\newblock {On the Convergence of a Non-linear Ensemble {Kalman} Smoother}.
\newblock {\em ArXiv e-prints}, November 2014.

\bibitem{LI20083574}
Jia Li and Dongbin Xiu.
\newblock On numerical properties of the ensemble {K}alman filter for data
  assimilation.
\newblock {\em Computer Methods in Applied Mechanics and Engineering},
  197(43):3574 -- 3583, 2008.
\newblock Stochastic Modeling of Multiscale and Multiphysics Problems.

\end{thebibliography}

\end{document}